\documentclass[10pt]{amsart}
\usepackage[utf8]{inputenc}
\usepackage{textalpha}

\usepackage{times}

\usepackage{mathtools, amsmath, amsthm, amssymb, bm, centernot, cancel, scalerel}
\usepackage{mathrsfs}
\usepackage{stmaryrd, algorithm, algpseudocode}
\usepackage{hyperref, caption, accents, enumitem}
\usepackage{multirow, multicol, tabularx}
\usepackage{subfigure}
\usepackage{tikz, pgfplots, pgfplotstable}
\usetikzlibrary{arrows, arrows.meta}
\usetikzlibrary{angles, quotes}
\usepackage{pst-node}
\usepackage{tikz-cd}
\usepackage[normalem]{ulem}

\DeclareMathAlphabet\mathbfcal{OMS}{cmsy}{b}{n}

    \usepackage[left=1in, right=1in, top=1.25in, bottom=1.25in]{geometry}
    \setlist[itemize]{leftmargin=*}
    \setlist[enumerate]{leftmargin=*}

    \usepackage[
    backend=biber,
    style=alphabetic,
    giveninits=true,
    sorting=nyt,
    maxnames=10
    ]{biblatex}
    \AtEveryBibitem{
        \clearfield{urldate}
        \clearfield{keywords}
        \clearfield{urlyear}
        \clearfield{urlmonth}
        \clearfield{url}
        \clearfield{issn}
        \clearfield{isbn}
        \clearfield{pages}
        \clearfield{doi}
    }
    
    \hypersetup{
        colorlinks=true, 
        linkcolor=blue,
        citecolor=red,
    }

\pgfplotsset{compat=1.18}

\newcommand{\rednote}[1]{{\color[hsb]{.0,1.0,.75}#1}}

    \newtheorem{theorem}{Theorem}[section]
    
    \newtheorem{corollary}[theorem]{Corollary}
    \newtheorem{lemma}[theorem]{Lemma}
    \newtheorem{remark}[theorem]{Remark}
    \newtheorem{proposition}[theorem]{Proposition}

    \numberwithin{equation}{section}

    \DeclareMathOperator{\tr}{\mathrm{tr}}

    \newcommand{\nablaG}{\nabla_\Gamma}
    
    \newcommand{\divG}{\mathrm{div}_\Gamma}
    \newcommand{\DeltaG}{\Delta_\G}
    
    \newcommand{\bn}{\mathbf n}
    \newcommand{\bx}{\mathbf x}
    \newcommand{\by}{\mathbf y}
    \newcommand{\bE}{\mathbf E}

    \newcommand{\bu}{\mathbf u}
    \newcommand{\bv}{\mathbf v}
    \newcommand{\bw}{\mathbf w}
    \newcommand{\bg}{\mathbf g}

    \newcommand{\bp}{\mathbf p}
    \newcommand{\ba}{\mathbf a}
    
    \newcommand{\bk}{\mathbf k}
    \newcommand{\be}{\mathbf e}
    \renewcommand{\bf}{\mathbf f}

    \newcommand{\bL}{\mathbf L}
    \newcommand{\bH}{\mathbf H}
    
    \newcommand{\bW}{\mathbf W}
    \newcommand{\bP}{\mathbf P}

    \newcommand{\bC}{\mathbf C}
    
    \newcommand{\bR}{\mathbf R}

    \newcommand{\bbR}{\mathbb R}

    \newcommand{\bbV}{\mathbb{V}}

    \newcommand{\ds}{\mathrm{d}s}

    \newcommand{\cV}{\mathcal{V}}

    \newcommand{\cT}{\mathcal{T}}

    \newcommand{\sT}{\mathscr{T}}
    
    \newcommand{\sD}{\mathscr{D}}
    
    \newcommand{\bcV}{\mathbfcal{V}}

    \newcommand{\embeds}{\hookrightarrow}

    \newcommand{\G}{\Gamma}
    
    \newcommand{\bI}{\mathbf{I}}

    \newcommand{\sA}{\mathscr{A}}
\title{Surface Stokes Without Inf-Sup Condition}

\author{Ricardo H. Nochetto and Mansur Shakipov}

\address{Department of Mathematics, University of Maryland, College Park, MD 20742}
\email{\{rhn,shakipov\}@umd.edu}

\subjclass[2020]{35J47, 35J50, 35M30, 35Q35, 58J05, 65N12, 65N15, 65N30}
\keywords{surface Stokes, elliptic system, well-posedness, parametric FEM, discrete stability, quasi-optimality, optimal error estimates}

\date{\today}

\begin{document}

\begin{abstract}
    For a $d$-dimensional hypersurface of class $C^3$ without boundary, we reformulate the surface Stokes equations as a nonsymmetric indefinite elliptic problem governed by two Laplacians. We then use this elliptic reformulation as a basis for a numerical method based on lifted parametric FEM. Assuming no geometric error for simplicity, we prove its well-posedness, quasi-best approximation in a robust mesh-dependent $H^1$-norm for any polynomial degree, as well as an optimal $L^2$ error estimate for both velocity and pressure. This entails a sufficiently small mesh size that solely depends on the Weingarten map and circumvents the usual discrete inf-sup condition. We present numerical experiments for velocity-pressure pairs with equal and disparate polynomial degrees, demonstrating that the proposed method is both accurate and practical.
\end{abstract}

\maketitle

\section{Introduction}
    The surface Stokes equations are an essential tool in modeling fluid flow on thin films and membranes \cite{Scriven1960, ArroyoDeSimone2009, RangamaniAgrawalMandadapuOsterSteigmann2013, SlatterySagis2013, Brenner2013, RahimiDeSimoneArroyo2013, KobaLiuGiga2017, JankuhnOlshanskiiReusken2018, ZhuSaintillanChern2025}. Discretization of the surface Stokes equations usually boils down to either a (Lagrangian) parametric FEM \cite{Dziuk1988, DziukElliott2013} or an (Eulerian) TraceFEM \cite{OlshanskiiReuskenGrande2009, OlshanskiiReusken2017, BurmanHansboLarsonMassing2018}. Similarly to the Euclidean case, both methods generally require the validity of an appropriate discrete inf-sup condition (also known as the Babu\v{s}ka-Brezzi condition) in order to ensure stability at the discrete level. This is a nontrivial matter both practically -- necessitating particular choices of finite element spaces for the velocity and the pressure -- and analytically -- requiring an often technical and sophisticated numerical analysis. There has been an increasing interest in numerical methods for the Surface Stokes and Navier-Stokes problems posed on either a stationary \cite{Fries2018, OlshanskiiQuainiReuskenYushutin2018, OlshanskiiYushutin2019, BrandnerReusken2020, BonitoDemlowLicht2020, LedererLehrenfeldSchoberl2020, JankuhnOlshanskiiReuskenZhiliakov2021, OlshanskiiReuskenZhiliakov2021, BrandnerJankuhnPraetorius2022, DemlowNeilan2024, HarderingPraetorius2024, DemlowNeilan2025-2, Reusken2025, ElliottMavrakis2025} or an evolving surface \cite{OlshanskiiReuskenZhiliakov2022, OlshanskiiReuskenSchwering2024}. As one would expect, all of these papers either assume or prove a discrete inf-sup condition for the primitive variable formulation.
    
    In the present paper, we show that for the tangential Stokes problem posed on a closed, sufficiently regular hypersurface, the need for the discrete inf-sup condition can be circumvented completely! To fix ideas, let $\G \subset \bbR^{d+1}, d \geq 2$ be a compact and connected $d$-dimensional manifold of class $C^3$ without boundary. Let $\bn$ be its outward pointing unit normal vector and $\bP := \bI - \bn \otimes \bn$ be the projection matrix. For a nonnegative integer $k$, let $H^k_\#$ be the scalar-valued Sobolev spaces $H^k(\G)$ with vanishing mean over $\G$, and let $\bH^k_t$ be the vector-valued Sobolev spaces $\bH^k(\G)$ that are a.e. tangential to $\G$. 
    Consider the tangent Stokes problem: given $(\bf, f) \in (\bH^1_t)' \times L^2_\#$, find a velocity-pressure pair $(\bu, u) \in \bH^1_t \times L^2_\#$ such that
    \begin{equation} \label{eq:stokes-intro}
        \begin{cases}
            (\nablaG \bu, \nablaG \bv) - (u, \divG \bv) & \kern-0.72em = \langle \bf, \bv \rangle_{(\bH^1_t)' \times \bH^1_t}, \qquad \forall \bv \in \bH^1_t, \\
            (\divG \bu, v) & \kern-0.72em = (f, v), \qquad \qquad \quad \ \ \ \forall v \in L^2_\#;
        \end{cases}
    \end{equation}
    this nonstandard notation will be convenient later. Here, $\nablaG, \divG := \tr(\nablaG)$ are, respectively, the tangential (covariant) gradient and tangential divergence; we recall that $\nablaG \bu = \bP (\nablaG^T u_i)_{i=1}^{d+1}$ where $\bu = (u_i)_{i=1}^{d+1}$. The parentheses in \eqref{eq:stokes-intro} denote the $L^2(\G)$-inner product. Even though the leading order term in our formulation \eqref{eq:stokes-intro} is the Bochner-Laplacian $-\Delta_B := -\bP \divG \nablaG$, with minor modifications for $d=2$, due to the Korn's inequality, our arguments still work when $\Delta_B$ is replaced with $\alpha \bP - \Delta_S, \alpha > 0$, where $\Delta_S := \bP \divG(\nablaG + \nablaG^T)$ is the more physically realistic vector Laplace operator; see \cite{KobaLiuGiga2017, JankuhnOlshanskiiReusken2018}. It is well-known that the standard weak formulation \eqref{eq:stokes-intro} is well-posed \cite{JankuhnOlshanskiiReusken2018, Reusken2020, HansboLarsonLarsson2020, BenavidesNochettoShakipov2025}.
    In the present paper, we show that the system above can be equivalently written as the following elliptic problem:
    \begin{equation} \label{eq:stokes-intro-new}
        \begin{cases}
            (\nablaG \bu, \nablaG \bv) - (u, \divG \bv) & \kern-0.72em = \langle \bf, \bv \rangle_{(\bH^1_t)' \times \bH^1_t}, \qquad \qquad \qquad \qquad \ \forall \bv \in \bH^1_t, \\
            (u, - \DeltaG v) - (\bR \bu, \nablaG v) & \kern-0.72em = \langle \bf, \nablaG v \rangle_{(\bH^1_t)' \times \bH^1_t} - (f, \DeltaG v), \qquad \forall v \in H^2_\#,
        \end{cases}
    \end{equation}
    where $\bR$ is the Ricci curvature map given by $\bR := \tr(\bW) \bW - \bW^2$ with $\bW := \nablaG \bn$ being the Weingarten map. To see this, we first note the following elementary identity valid for $\G$ of class $C^3$ and proven in \cite[Lemma 2.1]{BenavidesNochettoShakipov2025}: for all $\bv \in \bH^2_t, v \in H^2_\#$,
    \begin{equation} \label{eq:DeltaB-iden}
        \begin{split}
            (\nablaG \bv, \nablaG \nablaG v) & = (\nablaG^T \bv, \nablaG \nablaG v) = -(\bP \divG \nablaG^T \bv, \nablaG v) \\
            & = -(\nablaG \divG \bv + \bR \bv, \nablaG v) = (\divG \bv, \DeltaG v) - (\bR \bv, \nablaG v).
        \end{split}
    \end{equation}
    Here, the first equality follows from the symmetry of the covariant Hessian $\nablaG \nablaG v$ \cite[Proposition C.5]{BenavidesNochettoShakipov2025}, while the third equality is due to \cite[Lemma 2.1]{JankuhnOlshanskiiReusken2018} and density. Equality \eqref{eq:DeltaB-iden} can be further extended to $\bv \in \bH^1_\#$, again by density. We next test the momentum equation in \eqref{eq:stokes-intro} with $\bv = \nablaG v, v \in H^2_\#$ and integrate over $\G$:
    \[
        (\nablaG \bu, \nablaG \nablaG v) - (u, \DeltaG v) = \langle \bf, \nablaG v \rangle_{(\bH^1_t)' \times \bH^1_t}.
    \]
    Upon invoking \eqref{eq:DeltaB-iden} together with $\divG \bu = f$, we arrive at
    \[
        (u, - \DeltaG v) - (\bR \bu, \nablaG v) = \langle \bf, \nablaG v \rangle_{(\bH^1_t)' \times \bH^1_t} - (f, \DeltaG v),
    \]
    which is precisely the second equation in \eqref{eq:stokes-intro-new}. This chain of equalities can be reverted, thus giving the asserted equivalence between \eqref{eq:stokes-intro} and \eqref{eq:stokes-intro-new}. We thus focus our attention on the following general elliptic problem: given $(\bg, g) \in (\bH^1_t)' \times (H^1_\#)'$, find $(\bw, w) \in \bH^1_t \times H^1_\#$ such that for all $(\bv, v) \in \bH^1_t \times H^1_\#$:
    \begin{equation} \label{eq:ell-intro}
        \begin{cases}
            (\nablaG \bw, \nablaG \bv) + (\nablaG w, \bv) & \kern-0.72em = \langle \bg, \bv \rangle_{(\bH^1_t)' \times \bH^1_t}, \\
            (\nablaG w, \nablaG v) - (\bR \bw, \nablaG v) & \kern-0.72em = \langle g, v \rangle_{(H^1_\#)' \times H^1_\#}.
        \end{cases}
    \end{equation}
    In contrast to \eqref{eq:stokes-intro}, \eqref{eq:ell-intro} is a system of two Laplace equations on $\G$ perturbed by lower-order terms. The system \eqref{eq:ell-intro} induces an indefinite bilinear form on $\bH^1_t \times H^1_\#$ that satisfies a G\r{a}rding's inequality. By establishing a relation between \eqref{eq:stokes-intro} and \eqref{eq:ell-intro}, we prove that the latter system is well-posed. Moreover, we derive regularity estimates for \eqref{eq:ell-intro} and its adjoint.

    \textit{The well-posedness of \eqref{eq:ell-intro} and the validity of a G\r{a}rding's inequality make \eqref{eq:ell-intro} amenable to discretization with any pair of finite element spaces, thereby circumventing the Babu\v{s}ka-Brezzi inf-sup condition needed for \eqref{eq:stokes-intro}.} To explore this assertion, we choose the simplest possible discretization, namely, the lifted parametric finite element method \cite{Dziuk1988, DziukElliott2013}. This scheme ignores the geometric error due to approximating $\G$, yet conveys all essential features. We employ ideas by A. Schatz \cite{Schatz1974} to prove the discrete well-posedness, quasi-best approximation property for the pair $(\bw, w)$ in $H^1$, along with optimal $L^2$-error estimates. 
    
    Attempts to avoid the discrete inf-sup condition have a long history for bounded flat domains \cite{BrezziPitkaranta1984, HughesFranca1986, BrezziDouglas1988, DuranNochetto1989, DahlkeHochmuthUrban2000, BanschMorinNochetto2002, KondratyukStevenson2008, BoffiBrezziFortin2013}. They usually perturb the PDE system, and so are inconsistent, or add local stabilization (bubbles, least squares) to the original saddle point system. Our approach is entirely different in that the reformulation \eqref{eq:ell-intro} is globally consistent with \eqref{eq:stokes-intro}, yet allows any polynomial degrees for the pair $(\bw, w)$.

    The rest of the paper is organized as follows. In Section \ref{sec:prelims}, we introduce Sobolev spaces on $\G$, cite results from the elliptic and Stokes theory, and derive the crucial identity \eqref{eq:DeltaB-iden}. Sections \ref{sec:reformulation} and \ref{sec:discretization} deal with the analysis and discretization of problem \eqref{eq:ell-intro}, respectively. Section \ref{sec:numer-experiments} presents two numerical experiments, one illustrating the flexibility in choosing polynomial degrees for $(\bw, w)$ while maintaining stability. We close the paper with conclusions in Section \ref{sec:conclusions}.

\section{Preliminaries} \label{sec:prelims}
    Let $\G \subset \bbR^{d+1}, d \geq 2$ be a compact, connected $d$-dimensional stationary hypersurface without boundary of class $C^3$. Let $\bn$ be its outward unit normal vector, $\bP := \bI - \bn \otimes \bn$ be the projection matrix. Let $\nablaG$ be the covariant derivative operator and $\bW := \nablaG \bn$ be the Weingarten map. The following geometric quantity, called the Ricci curvature map, will play a crucial role in the forthcoming development:
    \begin{equation} \label{eq:R}
        \bR := \tr(\bW) \bW - \bW^2.
    \end{equation}
    This quantity plays an important role in relating different definitions of vector Laplacians on surfaces, see \cite{JankuhnOlshanskiiReusken2018, OlshanskiiReuskenZhiliakov2021, BenavidesNochettoShakipov2025}.
    \subsection{Function spaces}
        Let $W^{k,p}$ and $\bW^{k,p}$ denote, respectively, the standard scalar and vector-valued Sobolev spaces over $\G$.
        We introduce the following subspaces of $L^2$ and $\bL^2$:
        \begin{equation}\label{eq:spaces-l2}
            \begin{split}
                L^2_\# & := \left\{v \in L^2 : \int_\G v = 0 \right\}, \\
                \bL^2_t & := \{\bv \in \bL^2: \bv \cdot \bn = 0 \text{ a.e. on $\G$} \},
            \end{split}
        \end{equation}
        their corresponding $H^k, \bH^k, k \geq 1$-versions:
        \begin{equation}\label{eq:spaces-hm}
            \begin{split}
                H^k_\# & := H^k \cap L^2_\#, \\
                \bH^k_t & := \bH^k \cap \bL^2_t,
            \end{split}
        \end{equation}
        and the dual spaces
        \begin{equation}\label{eq:spaces-dual}
            \begin{split}
                H^{-k}_\# & := (H^k_\#)', \\
                \bH^{-k}_t & := (\bH^k_t)'.
            \end{split}
        \end{equation}
        For the Stokes problem, it is often convenient to consider the system for a pair of two variables $V := (\bv, v)$. For this reason, we also introduce the following spaces:
        \begin{equation}\label{eq:V-space}
            \bbV_{r, s} := \bH^r_t \times H^s_\#, \qquad \text{$r,s$ integer}.
        \end{equation}
        Of course, we have the relation $\bbV_{r,s}' \cong \bbV_{-r,-s}$. We shall use the convention $V = (\bv, v)$ consistently throughout this paper, which means that the solution to the Stokes system shall be denoted $U = (\bu, u)$ with $\bu$ being the velocity and $u$ being the pressure. This is not a standard notation, but it will simplify the presentation of the forthcoming analysis.
    
    \subsection{Bilinear forms}
        Let $(\cdot, \cdot)$ denote the $L^2$-inner product over $\G$. We shall use the same notation for the scalar-, vector-, or tensor-valued cases. Next, we define the Dirichlet forms:
        \begin{equation} \label{eq:dir-forms}
            \begin{split}
                a(u, v) & := (\nablaG u, \nablaG v), \qquad \forall u, v \in H^1_\#, \\
                \ba(\bu, \bv) & := (\nablaG \bu, \nablaG \bv), \qquad \forall \bu, \bv \in \bH^1_t.
            \end{split}
        \end{equation}
        For a positive integer $k$, we let $\langle \cdot, \cdot \rangle_{-k, k}$ denote the duality pairing between $H^{-k}_\#$ and $H^k_\#$ or between $\bH^{-k}_t$ and $\bH^k_t$.
    
    \subsection{Norms}
        For an integer $k$, we shall use $\|\cdot\|_k$ to mean $H^k$- or $\bH^k$-norm, and $\|\cdot\| := \|\cdot\|_0$. The $L^\infty$-norm would be denoted $\|\cdot\|_{0,\infty}$.
        
        We shall endow $L^2_\#$ and $\bL^2_t$ with the canonical norms, while the norms on $H^1_\#$ and $\bH^1_t$ are induced, respectively, by $a(\cdot, \cdot)$ and $\ba(\cdot, \cdot)$ due to the following Poincar\'e inequalities for a compact and connected $\G$ of class $C^3$ \cite{DziukElliott2013, BonitoDemlowNochetto2020, HansboLarsonLarsson2020, BenavidesNochettoShakipov2025}:
        \begin{equation} \label{eq:poincare}
            \begin{split}
                \|v\| & \leq C_{P,\#} \ a(v, v)^{1/2}, \qquad \forall v \in H^1_\#, \\
                \|\bv\| & \leq C_{P,t} \ \ba(\bv, \bv)^{1/2}, \qquad \forall \bv \in \bH^1_t.
            \end{split}
        \end{equation}
        % We shall thus also drop the zeroth-order part from the $H^k$- and $\bH^k$-norms for $k \geq 2$.
        We endow $\bbV_{r,s}$ with the graph norm:
        \begin{equation}\label{eq:V-space-norm}
            \|V\|_{r,s} := \left( \|\bv\|_r^2 + \|v\|_s^2 \right)^{1/2}.
        \end{equation}
    
    \subsection{Elliptic theory}
        We shall next consider two classical elliptic problems:
        \begin{itemize}
            \item \textit{Laplace-Beltrami}. Given $f \in H^{-1}_\#$, find $u \in H^1_\#$ such that for all $v \in H^1_\#$
            \begin{equation} \label{eq:lb}
                a(u, v) = \langle f, v \rangle_{-1,1}.
            \end{equation}
            \item \textit{Bochner-Laplace}. Given $\bf \in \bH^{-1}_t$, find $\bu \in \bH^1_t$ such that for all $\bv \in \bH^1_t$
            \begin{equation} \label{eq:bl}
                \ba(\bu, \bv) = \langle \bf, \bv \rangle_{-1,1}.
            \end{equation}
        \end{itemize}
        The well-posedness of \eqref{eq:lb} and \eqref{eq:bl} hinges, respectively, on the two Poincar\'e inequalities from \eqref{eq:poincare}. Higher-regularity results for \eqref{eq:lb} can be found in  \cite{DziukElliott2013, BonitoDemlowNochetto2020, BenavidesNochettoShakipov2025}, and for \eqref{eq:bl} in \cite{BenavidesNochettoShakipov2025}. We summarize these results in the following proposition.
        \begin{proposition}[elliptic theory] \label{prop:elliptic-theory}
            Let $\G$ be of class $C^3$. Then the Laplace-Beltrami \eqref{eq:lb} and the Bochner-Laplace \eqref{eq:bl} problems are well-posed, and there exist constants $C_{LB, 1}, C_{BL, 1} > 0$ depending only on $\G$ such that
            \[
            \begin{split}
                \|u\|_1 & \leq C_{LB, 1} \|f\|_{-1}, \\
                \|\bu\|_1 & \leq C_{BL, 1} \|\bf\|_{-1}.
            \end{split}
            \]
            If $f \in L^2_\#, \bf \in \bL^2_t$, then the solutions $u \in H^1_\#$ of \eqref{eq:lb} and $\bu \in \bH^1_t$ of \eqref{eq:bl} attain higher regularity, namely, $u \in H^2_\#, \bu \in \bH^2_t$ and there exist constant $C_{LB, 2}, C_{BL, 2} > 0$ depending only on $\G$ such that
            \[
            \begin{split}
                \|u\|_2 & \leq C_{LB, 2} \|f\|, \\
                \|\bu\|_2 & \leq C_{BL, 2} \|\bf\|.
            \end{split}
            \]
            % Furthermore, both equations are satisfied strongly a.e. on $\G$: 
            % \[
            % \begin{cases}
            %     - \DeltaG u & \kern-0.72em = f, \\
            %     - \Delta_B \bu & \kern-0.72em = \bf.
            % \end{cases}
            % \]
        \end{proposition}

    \subsection{The Stokes problem}
        We start with the following formulation of the surface Stokes problem: given $F = (\bf, f) \in \bbV_{-1,0}$, find $U = (\bu, u) \in \bbV_{1,0}$ such that for all $V = (\bv, v) \in \bbV_{1,0}$:
        \begin{equation} \label{eq:stokes}
            \begin{cases}
                \ba(\bu, \bv) - (u, \divG \bv) & \kern-0.72em = \langle \bf, \bv \rangle_{-1,1}, \\
                (\divG \bu, v) & \kern-0.72em = (f, v).
            \end{cases}
        \end{equation}
        The fact that this system is well-posed follows from Poincar\'e inequality \eqref{eq:poincare} and the following standard result proven in \cite{JankuhnOlshanskiiReusken2018, BenavidesNochettoShakipov2025}.
        \begin{lemma}[surjectivity of $\divG$] \label{lem:surj}
            Let $\G$ be of class $C^2$. Then $\divG : \bH^1_t \to L^2_\#$ is surjective.
        \end{lemma}
        
        A higher regularity result is proven in \cite{BenavidesNochettoShakipov2025} (also in \cite{OlshanskiiReuskenZhiliakov2021} for $d = 2$). We combine them in the following proposition.
        \begin{proposition}[Stokes] \label{prop:stokes}
            Let $\G$ be of class $C^3$. Then the surface Stokes problem \eqref{eq:stokes} is well-posed, and there exists $C_1 > 0$ depending only on $\G$ such that
            \[
            \|U\|_{1,0} \leq C_1 \|F\|_{-1,0}.
            \]
            If $F \in \bbV_{0,1}$, then the solution $U \in \bbV_{1,0}$ of \eqref{eq:stokes} attains higher regularity, namely, $U \in \bbV_{2,1}$ and there exists a constant $C_2 > 0$ depending only on $\G$ such that
            \[
            \|U\|_{2,1} \leq C_2 \|F\|_{0,1}.
            \]
        \end{proposition}
        
    \subsection{Decomposition of the Bochner Laplacian}
        The following lemma from \cite[Lemma 2.1]{BenavidesNochettoShakipov2025} is integral to the new method. The result below can be seen as a generalization to any dimension of the vector Laplacian decomposition trick used in \cite[Lemma 2.1]{OlshanskiiReuskenZhiliakov2021} to prove $H^2$-regularity of the surface Stokes problem for $d=2$.
        \begin{lemma}[decomposition of the Bochner Laplacian] \label{lem:stokes-iden}
            If $\G$ is of class $C^3$, then
            \begin{equation}
                \begin{split}
                    \ba(\bv, \nablaG q) & = (\divG \bv, \DeltaG q) - (\bR \bv, \nablaG q), \quad \forall \bv \in \bH^1_t, q \in H^2_\#,
                \end{split}
            \end{equation}
            where $\bR = \tr(\bW)\bW - \bW^2$ with $\bW = \nablaG \bn$ being the Weingarten map.
        \end{lemma}
        \begin{proof}
            Let $\bv \in \bC^2_t := \{\bv \in \bC^2(\G) : \bv \cdot \bn = 0\}, \phi \in C^2_\# := \{v \in C^2(\G) : (v, 1) = 0\}$. Due to the symmetry of the covariant Hessian \cite[Proposition C.5]{BenavidesNochettoShakipov2025},
            \[(\nablaG \bv, \nablaG \nablaG \phi) = (\nablaG^T \bv, \nablaG \nablaG \phi) = -(\bP \divG \nablaG^T \bv, \nablaG \phi).\]
            From \cite[Lemma 2.1]{JankuhnOlshanskiiReusken2018}, we know that for all $\bv \in \bC^2_t$
            \begin{equation} \label{eq:jor18}
                \bP \divG \nablaG^T \bv = \nablaG \divG \bv + \bR \bv.
            \end{equation}
            Hence,
            \[
            (\nablaG \bv, \nablaG \nablaG \phi) = -(\nablaG \divG \bv + \bR \bv, \nablaG \phi).
            \]
            Integrating the right-hand side by parts yields
            \[
            (\nablaG \bv, \nablaG \nablaG \phi) = (\divG \bv, \DeltaG v) - (\bR \bv, \nablaG \phi).
            \]
            The assertion then follows by the density of $\bC^2_t$ in $\bH^1_t$ and $C^2_\#$ in $H^2_\#$.
        \end{proof}
    
\section{Reformulation and stability} \label{sec:reformulation}
    For the rest of the paper, we will always assume that $\G$ is of class $C^3$ so that the higher-regularity result $U \in \bbV_{2,1}$ for the Stokes problem (Proposition \ref{prop:stokes}) is valid.
    \subsection{Reformulations of the momentum equation}
        Using Lemma \ref{lem:stokes-iden} (decomposition of the Bochner Laplacian), we can test the momentum equation in \eqref{eq:stokes} with $\bv = \nablaG v \in \bH^1_t$ for $v \in H^2_\#$ to arrive at
        \[
            (\divG \bu, \DeltaG v) - (\bR \bu, \nablaG v) - (u, \DeltaG v) = \langle \bf, \nablaG v \rangle_{-1,1}.
        \]
        Next, we recall that $\divG \bu = f$ to write
        \[
            (u, - \DeltaG v) - (\bR \bu, \nablaG v) = \langle \bf, \nablaG v \rangle_{-1,1} - (f, \DeltaG v),
        \]
        for all $v \in H^2_\#$. 

    \subsection{New Stokes system}
        Therefore, we obtain the following new elliptic system:
        \begin{equation} \label{eq:stokes-new}
            \begin{cases}
                \ba(\bu, \bv) - (u, \divG \bv) & \kern-0.72em = \langle \bg, \bv \rangle_{-1,1}, \\
                (u, - \DeltaG v) - (\bR \bu, \nablaG v) & \kern-0.72em = \langle g, v \rangle_{-2,2},
            \end{cases}
        \end{equation}
        for all $V = (\bv, v) \in \bbV_{1,2}$, where $G := (\bg, g) = \sD(\bf, f) \in \bbV_{-1,-2}$ for $(\bf, f) \in \bbV_{-1,0}$ is given by
        \begin{equation} \label{eq:sD}
            \bg = \bf, \qquad g = -\divG \bf - \DeltaG f.
        \end{equation}

    \subsection{Elliptic system} \label{sec:ell-reform}
        It turns out to be convenient to analyze the following general problem with stronger data than \eqref{eq:sD}: given $G = (\bg, g) \in \bbV_{-1,-1}$, find $W := (\bw, w) \in \bbV_{1,1}$ such that for all $V = (\bv, v) \in \bbV_{1,1}$:
        \begin{equation} \label{eq:ell}
            \begin{cases}
                \ba(\bw, \bv) + (\nablaG w, \bv) & \kern-0.72em = \langle \bg, \bv \rangle_{-1,1}, \\
                a(w, v) - (\bR \bw, \nablaG v) & \kern-0.72em = \langle g, v \rangle_{-1,1}.
            \end{cases}
        \end{equation}
        By defining the bilinear form $\sA : \bbV_{1,1} \times \bbV_{1,1} \to \bbR$ as
        \begin{equation} \label{eq:sA}
            \begin{split}
                \sA[W, V] & := \ba(\bw, \bv) + a(w, v) + (\nablaG w, \bv) - (\bR \bw, \nablaG v),
            \end{split}
        \end{equation}
        we can write \eqref{eq:ell} more concisely as follows: given $G \in \bbV_{-1,-1} \cong \bbV_{1,1}'$, find $W \in \bbV_{1,1}$ such that for all $V \in \bbV_{1,1}$
        \begin{equation} \label{eq:ell-pb}
            \sA[W, V] = G[V].
        \end{equation}
        The bilinear form $\sA$ induces a linear operator $\sT : \bbV_{1,1} \to \bbV_{-1,-1}$ given by
        \begin{equation} \label{eq:T}
            \langle \sT W, V \rangle_{\bbV_{-1,-1} \times \bbV_{1,1}} = \sA[W, V], \qquad \forall W, V \in \bbV_{1,1},
        \end{equation}
        whence $\sT$ can be formally written in strong form as follows
        \begin{equation}
            \sT W = \sT(\bw, w) = (- \Delta_B \bw + \nablaG w, -\DeltaG w + \divG (\bR \bw)). 
        \end{equation}
        The question of well-posedness of \eqref{eq:ell-pb} necessitates establishing whether the data map given by \eqref{eq:sD} is an isomorphism. In the next lemma, we shall prove that this is the case.
        \begin{lemma}[data map] \label{lem:data-map}
        If $\G$ is of class $C^2$, then the map $\sD$ induced by \eqref{eq:sD} is an isomorphism from $\bbV_{-1,0}$ to $\bbV_{-1,-2}$ and from $\bbV_{0,1}$ to $\bbV_{0,-1}$.
        \end{lemma}
        \begin{proof}
            For $\G$ of class $C^2$, $H^2$-regularity for the Laplace-Beltrami operator is valid \cite{DziukElliott2013, BonitoDemlowNochetto2020}. Therefore, $\DeltaG: H^2_\# \to L^2_\#$ is an isomophism, and by duality so is its adjoint $\DeltaG : L^2_\# \to H^{-2}_\#$. The forward map $(\bf, f) \mapsto (\bf, -\divG \bf - \DeltaG f)$ is continuous from $\bbV_{-1,0}$ to $\bbV_{-1,-2}$ because $\divG : \bH^{-1}_t \to H^{-2}_\#$ and $\DeltaG: L^2_\# \to H^{-2}_\#$ are continuous. Moreover, the inverse map is given by $(\bg, g) \mapsto (\bg, (-\DeltaG)^{-1}(\divG \bg + g))$ and is continuous from $\bbV_{-1,-2}$ to $\bbV_{-1,0}$ because $\DeltaG : L^2_\# \to H^{-2}_\#$ is an isomorphism and $\divG : \bH^{-1}_t \to H^{-2}_\#$ is continuous. This proves that $\sD : \bbV_{-1,0} \to \bbV_{-1,-2}$ is an isomorphism. The second result, namely that $\sD: \bbV_{0,1} \to \bbV_{0,-1}$ is also an isomorphism, follows by similar argumentation combined with the facts that $\DeltaG : H^1_\# \to H^{-1}_\#$ is an isomorphism and $\divG : \bL^2_t \to H^{-1}_\#$ is continuous. This concludes the proof.
        \end{proof}
        Using the newly proven result, we can now show that \eqref{eq:ell-pb} is well-posed.
        \begin{theorem}[elliptic reformulation] \label{thm:ell}
            Let $\G$ be of class $C^3$. For all $G = (\bg, g) \in \bbV_{-1,-1}$, there exists unique $W = (\bw, w) \in \bbV_{1,1}$ that solves \eqref{eq:ell-pb}. Furthermore, the following stability bound is valid for some constant $C_3 > 0$ depending only on $\G$:
            \begin{equation} \label{eq:ell-stab-in-1-1}
                \|W\|_{1,1} \leq C_3 \|G\|_{-1,-1}.
            \end{equation}
            Equivalently, the operator $\sT$ defined in \eqref{eq:T} is an isomorphism. 
            Moreover, if $G \in \bbV_{0,0}$, then the solution $W \in \bbV_{1,1}$ of \eqref{eq:ell-pb} admits higher regularity, namely $W \in \bbV_{2,2}$, and there exists constant $C_3' > 0$ depending only on $\G$ such that
            \begin{equation} \label{eq:ell-stab-in-2-2}
                \|W\|_{2,2} \leq C_3' \|G\|_{0,0}.
            \end{equation}
        \end{theorem}
        \begin{proof}
            The proof of the first result is done in two steps. By relating \eqref{eq:ell-pb} to Stokes \eqref{eq:stokes}, we can deduce that for all $G = (\bg, g) \in \bbV_{-1,-2}$ there exists unique $W = (\bw, w) \in \bbV_{1,0}$ that solves \eqref{eq:stokes-new} for arbitrary $G \in \bbV_{-1,-2}$. Then, assuming further that $G \in \bbV_{-1,-1} \subset \bbV_{-1,-2}$, we can improve the regularity of $w$ from $L^2_\#$ to $H^1_\#$. \bigskip

            \noindent \textbf{Step 1}: \textit{Existence of solution $W \in \bbV_{1,0}$}. By Lemma \ref{lem:data-map} (data map), given an arbitrary $G \in \bbV_{-1,-2}$, we can always construct an $F \in \bbV_{-1,0}$ isomorphic to $G$. Then, by construction, the solution $W = \tilde U \in \bbV_{1,0}$ of the Stokes problem \eqref{eq:stokes} with data $F$ solves the following elliptic problem with data $G$:
            \begin{equation} \label{eq:ell-pb-weaker}
                \begin{cases}
                    \ba(\bw, \bv) - (w, \divG \bv) & \kern-0.72em = \langle \bg, \bv \rangle_{-1,1}, \\
                    (w, - \DeltaG v) - (\bR \bw, \nablaG v) & \kern-0.72em = \langle g, v \rangle_{-2,2}.
                \end{cases}
            \end{equation}
            Moreover, $W \in \bbV_{1,0}$ depends continuously on $G \in \bbV_{-1,-2}$, namely, $C_3'' > 0$ depending only on $\G$ such that $\|W\|_{1,0} \leq C_3'' \|G\|_{-1,-2}$. \bigskip

            \noindent \textbf{Step 2}: \textit{Well-posedness of \eqref{eq:ell-pb}}. We next recall that $G \in \bbV_{-1,-1} \subset \bbV_{-1,-2}$ and $\|G\|_{-1,-2} \leq \|G\|_{-1,-1}$. By Step 1, we can construct a unique $W \in \bbV_{1,0}$ such that
            \begin{equation} \label{eq:bound-on-W-in-1-0}
                \|W\|_{1,0} \leq C_3'' \|G\|_{-1,-1}.
            \end{equation}
            Since $G \in \bbV_{-1,-1}$ and thus $g \in H^{-1}_\#$, we infer that for all $v \in H^2_\#$, $\langle g, v \rangle_{-2,2} = \langle g, v \rangle_{-1,1}$. Thus, the second equation in \eqref{eq:ell-pb-weaker} reads
            \[
            (w, - \DeltaG v) = \langle g, v \rangle_{-1,1} + (\bR \bw, \nablaG v), \qquad \forall v \in H^2_\#.
            \]
            This is an ultraweak formulation of the Laplace-Beltrami problem with the right-hand side in $H^{-1}_\#$ because $\bw \in \bH^1_t$ and $\bR$ is of class $C^1$ for $\G$ of class $C^3$. Proposition \ref{prop:elliptic-theory} (elliptic theory) ensures that $\DeltaG : H^1_\# \to H^{-1}_\#$ is an isomorphism, hence we can infer that $w \in H^1_\#$ and
            \begin{equation} \label{eq:bound-on-w-in-1}
                \begin{split}
                    \|w\|_{1} & \leq C_{LB,1} \|g - \divG (\bR \bw)\|_{-1} \leq C_{LB,1} (\|g\|_{-1} + \|\bR\|_{0,\infty} \|\bw\|) \\
                    & \leq C_3''' (\|g\|_{-1} + \|\bw\|_1) \leq C_3''' (\|G\|_{-1,-1} + \|W\|_{1,0}) \leq 2 C_3''' C_3'' \|G\|_{-1,-1},
                \end{split}
                \end{equation}
            where in the last step we used \eqref{eq:bound-on-W-in-1-0}. Since $w \in H^1_\#$, we can integrate terms involving $w$ in \eqref{eq:ell-pb-weaker} by parts, thereby yielding
            \[
            \begin{cases}
                \ba(\bw, \bv) + (\nablaG w, \bv) & \kern-0.72em = \langle \bg, \bv \rangle_{-1,1}, \\
                a(w, v) - (\bR \bw, \nablaG v) & \kern-0.72em = \langle g, v \rangle_{-1,1}.
            \end{cases}
            \]
            The left-hand side is precisely the definition \eqref{eq:sA} of $\sA$. Finally, combining \eqref{eq:bound-on-W-in-1-0} and \eqref{eq:bound-on-w-in-1}, we deduce that there exists $C_3 > 0$ depending only on $\G$ such that
            \[
            \|W\|_{1,1} \leq C_3 \|G\|_{-1,-1}.
            \]
            This concludes the proof of estimate \eqref{eq:ell-stab-in-1-1}. \bigskip
            
            \noindent \textbf{Step 3}: \textit{Higher regularity}. For the estimate \eqref{eq:ell-stab-in-2-2}, we assume that $G \in \bbV_{0,0}$. Since $G \in \bbV_{0,0} \subset \bbV_{-1,-1}$, from Step 2 we know that there exists unique $W \in \bbV_{1,1}$ that solves \eqref{eq:ell}. We first observe that $\bw \in \bH^1_t$ solves
            \[
            \ba(\bw, \bv) = (\bg - \nablaG w, \bv), \qquad \forall \bv \in \bH^1_t,
            \]
            with $\bg - \nablaG w \in \bL^2_t$. Hence, by Proposition \ref{prop:elliptic-theory} (elliptic theory), we infer that $\bw \in \bH^2_t$ and it depends continuously on $\bg - \nablaG w \in \bL^2_t$, and thus on $G \in \bbV_{0,0}$. The same argument can be repeated for $w \in H^1_\#$ that solves 
            \[
            a(w, v) = (g - \divG (\bR \bw), v), \qquad \forall v \in H^1_\#,
            \]
            thereby resulting in $w \in H^2_\#$ depending continuously on $G \in \bbV_{0,0}$. This yields \eqref{eq:ell-stab-in-2-2} and concludes the proof.
        \end{proof}
        The fact that problem \eqref{eq:ell-pb} is well-posed is a rather surprising result since $\sA$ is non-coercive in general. Since we arrived at the definition of $\sA$ by starting from the Stokes problem \eqref{eq:stokes}, it might be tempting to think that Theorem \ref{thm:ell} (elliptic reformulation) gives a super-regularity result for the Stokes problem \eqref{eq:stokes}, namely that $F = (\bf, f) \in \bbV_{-1,0}$ implies $U = (\bu, u) \in \bbV_{1,1}$. This, however, is not the case because the data map $\sD$ from \eqref{eq:sD} is not an isomorphism between $\bbV_{-1,0}$ and $\bbV_{-1,-1}$. With the same reasoning, $F = (\bf, f) \in \bbV_{0,1}$ would not yield $U = (\bu, u) \in \bbV_{2,2}$. What can be shown, however, is the following higher regularity result for Stokes which follows from Theorem \ref{thm:ell}, the data isomorphism $\sD$ in \eqref{eq:sD} and Lemma \ref{lem:stokes-iden} (decomposition of the Bochner Laplacian). We need two auxiliary spaces:
        \begin{equation} \label{eq:hdiv}
            \bH^{-1}_t(\divG) := \{\bv \in \bH^{-1}_t : \divG \bv \in H^{-1}_\#\}, \qquad \bH_t(\divG) := \{\bv \in \bL^2_t : \divG \bv \in L^2_\#\}.
        \end{equation}
        \begin{corollary}[higher regularity for Stokes] \label{cor:reg-res-stokes}
            Let $\G$ be of class $C^3$. If $F = (\bf, f) \in \bH^{-1}_t(\divG) \times H^1_\#$, then the solution $U = (\bu, u) \in \bbV_{1,0}$ of \eqref{eq:stokes} admits higher regularity, namely, $U \in \bbV_{1,1}$. Moreover, if $F = (\bf, f) \in \bH_t(\divG) \times H^2_\#$, then the solution $U = (\bu, u) \in \bbV_{1,0}$ of \eqref{eq:stokes} is in fact in $\bbV_{2,2}$.
        \end{corollary}
        \begin{proof}
            Since $G = (\bg, g) = \sD(\bf, f) = (\bf, -\divG \bf - \DeltaG f)$ from \eqref{eq:sD}, the regularity assumption $G \in \bbV_{-1,-1}$ is guaranteed by $\bf \in \bH^{-1}_t(\divG)$ and $f \in H^1_\#$. Then by Theorem \ref{thm:ell}, there exists unique $W = (\bw, w) \in \bbV_{1,1}$ such that
            \[
            \begin{cases}
                \ba(\bw, \bv) + (\nablaG w, \bv) & \kern-0.72em = \langle \bf, \bv \rangle_{-1,1}, \\
                a(w, v) - (\bR \bw, \nablaG v) & \kern-0.72em = \langle \bf, \nablaG v \rangle_{-1,1} - (f, \DeltaG v),
            \end{cases}
            \]
            for all $V = (\bv, v) \in \bbV_{1,1}$. We need to show that $\divG \bw = f$. Using Lemma \ref{lem:stokes-iden}, we infer that the second equation reads
            \[
            a(w, v) + \ba(\bw, \nablaG v) - (\divG \bw, \DeltaG v) = \langle \bf, \nablaG v \rangle_{-1,1} - (f, \DeltaG v).
            \]
            Upon recognizing that $\ba(\bw, \nablaG v) + a(w, v) = \langle \bf, \nablaG v \rangle_{-1,1}$ from the momentum equation tested with $\bv = \nablaG v, v \in H^2_\#$, we arrive at
            \[
            (\divG \bw, \DeltaG v) = (f, \DeltaG v), \qquad \forall v \in H^2_\#.
            \]
            By the $H^2$-regularity of the Laplace-Beltrami problem from Proposition \ref{prop:elliptic-theory} (elliptic theory), we have
            \[
            (\divG \bw, v) = (f, v), \qquad \forall v \in L^2_\#,
            \]
            whence $\divG \bw = f$ a.e. on $\G$. Therefore, $W = (\bw, w) \in \bbV_{1,1}$ satisfies the Stokes problem \eqref{eq:stokes}. This finishes the proof of the first result. The same argument can easily be adapted to prove the second statement. In particular, $\bf \in \bH_t(\divG)$ and $f \in H^2_\#$ ensure that $G \in \bbV_{0,0}$. Then, invoking the higher regularity result \eqref{eq:ell-stab-in-2-2} from Theorem \ref{thm:ell}, namely $(\bw, w) \in \bbV_{2,2}$, and proceeding along the same lines finishes the proof.
        \end{proof}
        
    \subsection{Adjoint problem}
        Theorem \ref{thm:ell} (elliptic reformulation) shows that the linear operator $\sT : \bbV_{1,1} \to \bbV_{-1,-1}$ defined in \eqref{eq:T} is an isomorphism. It is then a classical result from functional analysis \cite[\S 2.7]{Brezis2011} that the adjoint operator $\sT^*: \bbV_{1,1} \to \bbV_{-1,-1}$ is also an isomorphism. To get the explicit expression for $\sT^*$, we use the definition \eqref{eq:sA} of $\sA$ and integrate by parts formally to obtain
        \[
        \sA[W, V] = (\bw, - \Delta_B \bv - \bR \nablaG v) + (w, - \divG \bv - \DeltaG v),
        \]
        thereby yielding the strong form of $\sT^*$
        \[
        \sT^* V = \sT^*(\bv, v) = (- \Delta_B \bv - \bR \nablaG v, - \DeltaG v - \divG \bv).
        \]
        Given an arbitrary $G = (\bg, g) \in \bbV_{-1,-1}$, to write the weak formulation of the adjoint problem $\sT^* V = G$ of \eqref{eq:ell} within the space $\bbV_{1,1}$, we seek $V = (\bv, v) \in \bbV_{1,1}$ such that
        \begin{equation} \label{eq:adjoint}
            \begin{cases}
                \ba(\bv, \by) - (\bR \nablaG v, \by) & \kern-0.72em = \langle \bg, \by \rangle_{-1,1}, \\
                a(v, y) + (\bv, \nablaG y) & \kern-0.72em = \langle g, y \rangle_{-1,1},
            \end{cases}
        \end{equation}
        for all $Y := (\by, y) \in \bbV_{1,1}$. The next result immediately follows from the fact that $\sT^*$ is an isomorphism.
        \begin{corollary}[adjoint] \label{cor:adjoint}
            Let $\G$ be of class $C^3$. For all $G = (\bg, g) \in \bbV_{-1,-1}$, there exists unique $V = (\bv, v) \in \bbV_{1,1}$ that solves \eqref{eq:adjoint}. Furthermore, there exists $C_4 > 0$ depending only on $\G$ such that the following a priori estimate is valid:
            \[
            \|V\|_{1,1} \leq C_4 \|G\|_{-1,-1}.
            \]
        \end{corollary}
        The following Lemma will be instrumental in the Aubin-Nitsche duality argument of Section \ref{sec:duality-arg}, which is necessary to deduce the well-posedness of the problem at the discrete level. Its proof is very similar to the proof of \eqref{eq:ell-stab-in-2-2} in Theorem \ref{thm:ell} (elliptic reformulation) and will therefore be omitted.
        \begin{lemma}[higher regularity of the adjoint problem] \label{lem:adjoint-regularity}
            Let $\G$ be of class $C^3$. If $G = (\bg, g) \in \bbV_{0,0}$, then the solution $V = (\bv, v) \in \bbV_{1,1}$ of \eqref{eq:adjoint} admits higher regularity, namely, $V \in \bbV_{2,2}$, and there exists $C_5 > 0$ depending only on $\G$ such that
            \[
            \|V\|_{2,2} \leq C_5 \|G\|_{0,0}.
            \]
        \end{lemma}

    \subsection{G\r{a}rding's inequality}
        In view of Theorem \ref{thm:ell} (elliptic reformulation), the Banach-Ne\v{c}as theorem \cite[Theorem 2.6]{ErnGuermond2004} ensures that $\sA$ is continuous, the kernel of the adjoint problem is trivial, and the following inf-sup constant $c_\sA > 0$ depends only on $\G$:
        \begin{equation}
            c_\sA := \inf_{0 \neq W \in \bbV_{1,1}} \sup_{0 \neq V \in \bbV_{1,1}} \frac{\sA[W, V]}{\|W\|_{1,1} \|V\|_{1,1}} > 0.
        \end{equation}
        We shall next show that $\sA$ satisfies a G\r{a}rding's inequality, which will become crucial later for discretization. For a specific class of hypersurfaces $\G$, $\sA$ turns out to be coercive, as shown in the next section. Then the well-posedness of \eqref{eq:ell-pb} can be deduced by the Lax-Milgram lemma.
        \begin{proposition}[properties of $\sA$] \label{prop:prop-of-sA}
            The bilinear form $\sA$ defined in \eqref{eq:sA} is continuous and satisfies G\r{a}rding's inequality for all $V, W \in \bbV_{1,1}$
            \[
            \begin{split}
                |\sA[W, V]| & \leq C_{\sA} \|W\|_{1,1} \|V\|_{1,1}, \\
                c_{\sA, 1} \|W\|_{1,1}^2 - c_{\sA, 2} \|W\|_{0,0}^2 & \leq \sA[W, W],
            \end{split}
            \]
            where $c_{\sA,1}, c_{\sA,2}, C_{\sA} > 0$ are defined as
            \[
            c_{\sA,1} := \frac12, \quad c_{\sA,2} := \frac12 \|\bP - \bR\|_{0,\infty}^2, \quad C_{\sA} := 1 + C_{P,t} \max\{1,\|\bR\|_{0,\infty}\},
            \]
            with $C_{P,t}$ being the Poincar\'e constant for tangent vector fields from \eqref{eq:poincare}. 
        \end{proposition}
        \begin{proof}
            We first prove the G\r{a}rding's inequality using elementary estimates, namely
            \[
            \begin{split}
                \sA[W, W] & = \ba(\bw, \bw) + a(w, v) + (\nablaG w, \bw) - (\bR \bw, \nablaG w) \\
                & = \|\bw\|_{1}^2 + \|w\|_1^2 + (\nablaG w, (\bP - \bR) \bw) \\
                & \geq \|\bw\|_{1}^2 + \frac12 \|w\|_1^2 - \frac12 \|\bP - \bR\|_{0,\infty}^2 \|\bw\|^2 \\
                & \geq \frac12 \|W\|_{1,1}^2 - \frac12 \|\bP - \bR\|_{0,\infty}^2 \|W\|_{0,0}^2. 
            \end{split}
            \]
            We next verify the continuity of $\sA$
            \[
            \begin{split}
                |\sA[W, V]| & = |\ba(\bw, \bv) + a(w, v) + (\nablaG w, \bv) - (\bR \bw, \nablaG v)| \\
                & \leq \|\bw\|_{1} \|\bv\|_{1} + \|w\|_1 \|v\|_1 + |(\nablaG w, \bv)| + |(\bR \bw, \nablaG v)| \\
                & \leq \|\bw\|_{1} \|\bv\|_{1} + \|w\|_1 \|v\|_1 + \|w\|_1 C_{P,t} \|\bv\|_{1} + \|\bR\|_{0,\infty} C_{P,t} \|\bw\|_1 \|v\|_1.
            \end{split}
            \]
            Altogether, this yields
            \[
            |\sA[W, V]| \leq (1 + C_{P,t} \max\{1,\|\bR\|_{0,\infty}\}) \|W\|_{1,1} \|V\|_{1,1},
            \]
            and concludes the proof.
        \end{proof}
    
    \subsection{Special coercive cases}
        Even though we know that \eqref{eq:ell-pb} is well-posed for any data $G \in \bbV_{-1,-1}$ and any hypersurface $\G$ of class $C^3$, it is still meritorious to ask what assumption on $\G$ ensures that \eqref{eq:ell-pb} is a coercive problem.
        \begin{theorem}[coercivity for a perturbation of the $d$-sphere] \label{thm:ell-coercive}
            Suppose that
            \[
            \sigma := \|\bP - \bR\|_{0,\infty} C_{P,t} < 2.
            \]
            Then for each $G \in \bbV_{-1,-1}$ there a exists unique solution $W \in \bbV_{1,1}$ of \eqref{eq:ell-pb}. Moreover, the following stability estimate is valid:
            \[
            \|W\|_{1,1} \leq \frac{2}{2-\sigma} \|G\|_{-1,-1}.
            \]
        \end{theorem}
        \begin{proof}
            We know that $\sA$ is continuous from Proposition \ref{prop:prop-of-sA} (properties of $\sA$). 
            In order to obtain the sharpest coercivity constant, we shall deviate from the proof of G\r{a}rding's inequality from Proposition \ref{prop:prop-of-sA} slightly. We start with
            \[\sA[W, W] = \|\bw\|_{1}^2 + \|w\|_1^2 + (\nablaG w, (\bP - \bR) \bw).\]
            We estimate the last term using the definition of $\sigma$ together with the Poincar\'e inequality from \eqref{eq:poincare} for vector fields
            \[ \begin{split}
                |(\nablaG w, (\bP - \bR) \bw)| & \leq \|\nablaG w\| \|\bP - \bR\|_{0,\infty} \|\bw\| \\
                & \leq \frac{\sigma}{2} \|\nablaG w\|^2 + \frac{1}{2 \sigma} \|\bP - \bR\|_{0,\infty}^2 C_{P,t}^2 \|\nablaG \bw\|^2 \\
                & = \frac{\sigma}{2} \|\nablaG w\|^2 + \frac{\sigma^2}{2 \sigma} \|\nablaG \bw\|^2 = \frac{\sigma}{2} \|W\|_{1,1}^2.
            \end{split} \]
            Altogether, we obtain
            \[
            \sA[W, W] \geq \|W\|_{1,1}^2 - \frac{\sigma}{2}\|W\|_{1,1}^2 = \frac{2 - \sigma}{2} \|W\|_{1,1}^2.
            \]
            By the assumption that $\sigma < 2$, we infer coercivity of $\sA$. Thus, by the Lax-Milgram theorem, problem \eqref{eq:ell-pb} is well-posed, and we get the asserted stability estimate.
        \end{proof}
    
        \begin{corollary}[$\G = 2$-sphere]
            Let $\G$ be the 2-dimensional sphere of radius $r > 0$. Then \eqref{eq:ell-pb} is symmetrizable, admits a unique solution $W \in \bbV_{1,1}$ for every $G \in \bbV_{-1,-1}$, and the following stability bound is valid:
            \[
            \|\bw\|_1^2 + r^2 \|w\|_1^2 \leq \|\bg\|_{-1}^2 + r^2 \|g\|_{-1}^2.
            \]
        \end{corollary}
        \begin{proof}
            Since $\G$ is 2-dimensional, $\bR = K \bP$, where $K$ is the Gauss curvature \cite[Lemma A.3]{JankuhnOlshanskiiReusken2018}. For the sphere of radius $r > 0$, we have $K = \frac{1}{r^2}$. Hence, \eqref{eq:ell-pb} becomes
            \[
            \begin{cases}
                (\nablaG \bw, \nablaG \bv) + (\nablaG w, \bv) & \kern-0.72em = \langle \bg, \bv \rangle_{-1,1}, \\
                (\bw, \nablaG v) - r^2 (\nablaG w, \nablaG v) & \kern-0.72em = -r^2 \langle g, v \rangle_{-1,1}.
            \end{cases}
            \]
            This system is clearly symmetric and coercive because testing with $V = W$ yields
            \[ \|\nablaG \bw\|^2 + r^2 \|\nablaG w\|^2 = \langle \bg, \bw \rangle_{-1,1} + r^2 \langle g, w \rangle_{-1,1}. \]
            The desired stability estimate then immediately follows.
        \end{proof}

    \subsection{Perturbed surface diffusion operator}
        So far, we have focused our attention on the Bochner-Laplace operator $\Delta_B$ used for the velocity part due to its simpler mathematical structure, namely that its kernel is trivial. In this section, we will show that the elliptic reformulation of the Stokes problem \eqref{eq:ell} with the perturbed \textit{surface diffusion} operator $\alpha \bI - \Delta_S, \alpha > 0$, where $\Delta_S := \bP \divG (\nablaG + \nablaG^T)$ replaces $\Delta_B = \bP \divG \nablaG$, is also well-posed for $d=2$. Let $E_s := \nablaG + \nablaG^T$ be the symmetric covariant derivative operator.

        We note that the perturbed operator of type $\alpha \bI - \Delta_S, \alpha > 0$ is very common in the numerical analysis literature \cite{OlshanskiiYushutin2019, BonitoDemlowLicht2020, OlshanskiiReuskenZhiliakov2021, DemlowNeilan2024}. The most ``physical" version of such an operator would correspond to $\alpha = 0$, cf. \cite{KobaLiuGiga2017, JankuhnOlshanskiiReusken2018}. That being said, in such a case, the solution is unique only up to a Killing vector field (a vector field $\bv \in \bH^1_t$ with $E_s(\bv) = 0$). This poses a major challenge at the discrete level \cite{BonitoDemlowLicht2020}, which is why the zeroth-order term $\alpha \bu, \alpha > 0$ is often introduced. The perturbed operator with $\alpha > 0$ also serves as a prototype for the discrete-in-time transient Stokes problem. We stress that the analysis with $\alpha=0$ is possible provided that function spaces for the velocity are orthogonal to the kernel of $E_s$, cf. \cite{JankuhnOlshanskiiReusken2018}. We, however, do not pursue such a direction in this manuscript.
        
        Let us now define the bilinear form on $\bH^1_t$ induced by $\alpha \bI - \Delta_S$:
        \begin{equation}
            \ba_{s,\alpha}(\bu, \bv) := (E_s(\bu), E_s(\bv)) + \alpha(\bu, \bv).
        \end{equation}
        We consider the following Stokes problem: given $F = (\bf, f) \in \bbV_{-1,0}$, find $U = (\bu, u) \in \bbV_{1,0}$ such that for all $V = (\bv, v) \in \bbV_{1,0}$
        \begin{equation} \label{eq:stokes-2}
            \begin{cases}
                \ba_{s, \alpha}(\bu, \bv) - (u, \divG \bv) & \kern-0.72em = \langle \bf, \bv \rangle_{-1,1}, \\
                (\divG \bu, v) & \kern-0.72em = (f, v).
            \end{cases}
        \end{equation}
        The fact that \eqref{eq:stokes-2} is well-posed follows from Korn's inequality, proved in \cite[Lemma 4.1]{JankuhnOlshanskiiReusken2018} for $d=2$, and Lemma \ref{lem:surj} (surjectivity of the divergence $\divG$). The zeroth order term $\alpha \bu$ rules out the nontrivial kernel of $E_s$. Of course, the stability constant will scale as $1/\alpha$ for $\alpha > 0$ small. Upon invoking Lemma \ref{lem:stokes-iden} (decomposition of the Bochner-Laplacian) together with \eqref{eq:jor18}, we can rewrite \eqref{eq:stokes-2} as
        \begin{equation} \label{eq:stokes-2-new}
            \begin{cases}
                \ba_{s, \alpha}(\bu, \bv) - (u, \divG \bv) & \kern-0.72em = \langle \bg, \bv \rangle_{-1,1}, \\
                (u, - \DeltaG v) - (2 \bR \bu, \nablaG v) & \kern-0.72em = \langle g, v \rangle_{-2,2}.
            \end{cases}
        \end{equation}
        for all $V = (\bv, v) \in \bbV_{1,2}$, where $G = (\bg, g) \in \bbV_{-1,-2}$ is given by
        \begin{equation}\label{eq:sD-2}
            \bg = \bf, \qquad g = - \divG \bf + (\alpha I - 2\DeltaG) f.
        \end{equation}
        We are now essentially in the same situation as in (\ref{eq:stokes-new}, \ref{eq:sD}): $\alpha \bI - \Delta_S$ is coercive on $\bH^1_t$, and the data map $\sD$ induced by \eqref{eq:sD-2} is an isomorphism between $\bbV_{-1,0}$ and $\bbV_{-1,-2}$. The rest of the development of this paper applies to the perturbed surface diffusion operator $\alpha \bI - \Delta_S$ along the same lines. We now turn our attention back to the Stokes formulation with $\Delta_B$ as we dive into discretization.
        
\section{Discretization} \label{sec:discretization}
    We are now ready to discretize \eqref{eq:ell} using lifted parametric FEM \cite{Dziuk1988, DziukElliott2013}. One of the outstanding issues of the classical parametric FEM is geometric inconsistency. This is a non-trivial matter, which, however, appears to be well-understood; see \cite{Dziuk1988, DziukElliott2013, Demlow2009} for the Laplace-Beltrami problem and \cite{HansboLarsonLarsson2020} for the vector Laplace problems. For that reason, 
    in this paper, we assume geometric exactness and instead focus on carrying out the argument by \cite{Schatz1974} that allows proving discrete well-posedness for non-coercive systems. In the numerical experiment section, we demonstrate that under appropriate treatment of the geometry, optimal higher-order error rates are achieved.
    
    In this section, our goal is to derive error estimates in $H^1$ and $L^2$ for an arbitrary hypersurface $\G$, that is, without any assumption on how close $\G$ is to a hypersphere (as in Theorem \ref{thm:ell-coercive}). As was said earlier, we shall employ a trick by A. Schatz \cite{Schatz1974}, which in turn uses the Aubin-Nitsche duality argument \cite{Nitsche1970}, to obtain well-posedness at the discrete level. The analysis of the discrete problem in the coercive regime is straightforward.

    We would like to mention that other FEMs for surface PDEs can be employed, like TraceFEM \cite{OlshanskiiReuskenGrande2009, OlshanskiiReusken2017, BurmanHansboLarsonMassing2018, Reusken2022}. The forthcoming analysis easily extends to the stabilized TraceFEM discretization provided that the geometry is exact. The higher-order treatment of the geometry is realized via isoparametric mappings, see \cite{Lehrenfeld2016, GrandeLehrenfeldReusken2018, LehrenfeldReusken2018}.
    
    \subsection{Geometry}
        Let $\G \subset \bbR^{d+1}, d \geq 2$ be of class $C^3$ be given so that Proposition \ref{prop:stokes} (Stokes) is valid. Let $\widetilde \cT_h$ be a quasi-uniform and shape-regular partition of $\G$ into either flat or curved elements of size $h$. We call the discrete surface $\G_h$. Subordinate to $\widetilde \cT_h$, we define two $H^1$-conforming finite element spaces, $\widetilde \bcV_h \subset H^1(\G_h; \bbR^{d+1}), \widetilde \cV_h \subset H^1(\G_h)$. 
        
        We assume that $h > 0$ is small enough to resolve all principal curvatures of $\G$; this ensures that the closest point projection $\bp : \G_h \to \G$ given by $\bp(\bx) := \bx - d(\bx) \nabla d(\bx) $ is bijective, where $d$ is the distance function to $\G$ \cite[\S2.4]{BonitoDemlowNochetto2020}. We introduce the lifting operator $(\cdot)^l : L^2(\G_h) \to L^2(\G)$ defined as follows:
        \begin{equation} \label{eq:lift}
            v^l(\bx) := v(\bp^{-1}(\bx)), \qquad \forall v \in L^2(\G_h), \forall \bx \in \G.
        \end{equation}
        For vector-valued functions, we lift each component. The inverse lifting operator $(\cdot)^{-l} : L^2(\G) \to L^2(\G_h)$ is also well-defined, and we have the following classical equivalence results \cite{Dziuk1988, Demlow2009, BonitoDemlowNochetto2020} for all $v \in L^2(\G_h)$ and $\phi \in H^1(\G_h)$
        \begin{equation} \label{eq:lift-equiv}
            \|v^l\|_{L^2(\G)} \simeq \|v\|_{L^2(\G_h)}, \qquad \|\nabla_\G \phi^l\|_{L^2(\G)} \simeq \|\nabla_{\G_h} \phi\|_{L^2(\G_h)}.
        \end{equation}
        Let $\cT_h$ be the lift of the partition $\widetilde \cT_h$ from $\G_h$ to $\G$. Thus, the union of all $T \in \cT_h$ yields $\G$. Then, we define $\bcV_h$ to be the lifted FE space associated with $\widetilde \bcV_h$, and $\cV_h$ to be the lifted subspace of $\widetilde \cV_h$ satisfying $\cV_h \subset L^2_\#$. Finally, let $\bbV_h := \bcV_h \times \cV_h$.
    
    \subsection{Towards a penalized formulation}
        Discretizing the space $\bH^1_t$ is challenging due to the tangentiality constraint. For $C^0$-conforming finite element spaces defined on $C^0$-conforming triangulations of $\G$, attempts to impose this constraint strongly lead to locking. One popular option is to impose this constraint weakly by penalizing the normal component \cite{OlshanskiiQuainiReuskenYushutin2018, HansboLarsonLarsson2020}. This option gives rise to a rather simple analysis, which follows from favorable properties of the penalized system at the continuous level \cite{JankuhnOlshanskiiReusken2018}. The only disadvantage of this approach is that the term penalizing the normal component needs to have a higher-order approximation of the unit normal vector $\bn := \nabla d$ to $\G$ \cite{OlshanskiiQuainiReuskenYushutin2018, HansboLarsonLarsson2020, JankuhnOlshanskiiReuskenZhiliakov2021}, which may not always be available. Another possibility is to relax the $\bH^1$-conformity to $\bH(\divG)$, as was done in \cite{BonitoDemlowLicht2020, DemlowNeilan2024, DemlowNeilan2025-2}. This allows imposing the tangentiality constraint elementwise (thus circumventing the need for a higher-order approximation of the normal) at the expense of being non-$\bH^1$-conforming. We stress that the proposed elliptic reformulation \eqref{eq:ell} can be discretized either way. However, for the sake of simplifying the presentation, we opt to proceed via the penalty approach.
        
        Let $\bH^1_* := \overline{\bH^1}^{(\|\bP \cdot\|_1^2 + \|\bn \cdot\|^2)^{1/2}}$. For all $\bw, \bv \in \bH^1_*$, we define
        \begin{equation}
        \begin{split}
            \bk_\tau(\bw, \bv) & := \frac{1}{\tau^2} (\bw \cdot \bn, \bv \cdot \bn), \qquad \tau \leq 1.  \\
            \ba_\tau(\bw, \bv) & := \ba(\bP \bw, \bP \bv) + \bk_\tau(\bw, \bv).
        \end{split}
        \end{equation}
        We next show that $\ba_\tau$ induces a norm on $\bH^1_*$.
        \begin{lemma}[penalized Poincar\'e] \label{lem:pen-poincare}
            There exists $C_P'>0$ such that for all $\bv \in \bH^1_*$, $\|\bv\| \leq C_P' (\|\bP \bv\|_1^2 + \|\bn \cdot \bv\|)^{1/2}$.
        \end{lemma}
        \begin{proof}
            From Poincar\'{e}'s inequality for tangential vector fields \eqref{eq:poincare}, we deduce that
            \[
            \|\bv\|^2 = \|\bn \cdot \bv\|^2 + \|\bP \bv\|^2 \leq \max\{1,C_{P,t}^2\}(\|\bn \cdot \bv\|^2 + \|\bP \bv\|_1^2).
            \]
            This concludes the proof.
        \end{proof}
    
    \subsection{Bilinear forms}
        Besides the bilinear forms $a : \cV_h \times \cV_h \to \bbR$ and $\ba_h : \bcV_h \times \bcV_h \to \bbR$ for the second order operators, we next introduce the discrete bilinear forms corresponding to their compact perturbations
        \begin{equation} \label{eq:prop-of-b-c}
            \begin{split}
                b(w_h, \bv_h) & := (\nablaG w_h, \bv_h)_\G = b(w_h, \bP \bv_h), \\
                c(\bw_h, v_h) & := (\bR \bw_h, \nablaG v_h)_\G = c(\bP \bw_h, v_h),
            \end{split}
        \end{equation}
        because $\bR \bP = (\tr(\bW)\bW - \bW^2) \bP = \bR$. Let $\sA_h : \bbV_h \times \bbV_h \to \bbR$ be the discrete counterpart of the bilinear form $\sA$ from \eqref{eq:sA} and be defined as 
        \begin{equation} \label{eq:sAh}
            \begin{split}
                \sA_h[W_h, V_h] & := \ba_h(\bw_h, \bv_h) + b(w_h, \bv_h) + a(w_h, v_h) - c(\bw_h, v_h).
            \end{split}
        \end{equation}
        Introducing the following notation:
        \begin{equation} \label{eq:Vht}
            V_{h,t} := (\bP \bv_h, v_h), \qquad \forall V_h \in \bbV_h,
        \end{equation}
        we can conveniently write \eqref{eq:sAh} as
        \begin{equation} \label{eq:sAh-2}
            \begin{split}
                \sA_h[W_h, V_h] & = \sA[W_{h,t}, V_{h,t}] + \bk_h(\bw_h, \bv_h).
            \end{split}
        \end{equation}
        
    \subsection{Norms}
        We endow $\cV_h$ with the $a(\cdot, \cdot)^{1/2}$-norm, and $\bcV_h$ with:
        \begin{equation} \label{eq:1h-norm}
            \|\bv_h\|_{1h} := \ba_h(\bv_h, \bv_h)^{1/2}.
        \end{equation}
        Note that \eqref{eq:1h-norm} is indeed a norm by Lemma \ref{lem:pen-poincare} (penalized Poincar\'e). We recall that $\bbV_h = \bcV_h \times \cV_h$, and endow it with the following graph norm: 
        \begin{equation}
            \begin{split}
                \|V_h\|_{1h,1}^2 & := \|\bv_h\|_{1h}^2 + \|v_h\|_1^2 = \|\bP \bv_h\|_1^2 + \bk_h(\bv_h, \bv_h) + \|v_h\|_1^2 \\
                & \ = \|V_{h,t}\|_{1,1}^2 + \bk_h(\bv_h, \bv_h), \qquad \forall V_h = (\bv_h, v_h) \in \bbV_h.
            \end{split} 
        \end{equation}
        
    \subsection{The finite element method}
        Given $G = (\bg, g) \in \bbV_{-1,-1}$, find $W_h := (\bw_h, w_h) \in \bbV_h$ such that
        \begin{equation} \label{eq:discrete}
            \begin{cases}
                \ba_h(\bw_h, \bv_h) + b(w_h, \bv_h) & \kern-0.72em = \langle \bg, \bP \bv_h \rangle_{-1,1}, \\
                a(w_h, v_h) - c(\bw_h, v_h) & \kern-0.72em = \langle g, v_h \rangle_{-1,1},
            \end{cases}
        \end{equation}
        for all $V_h := (\bv_h, v_h) \in \bbV_h$. This problem reads equivalently as follows in terms of $\sA_h$: given $G \in \bbV_{-1,-1}$, find $W_h \in \bbV_h$ such that for all $V_h \in \bbV_h$
        \begin{equation} \label{eq:discrete-compact}
            \sA_h[W_h, V_h] = G[V_{h,t}].
        \end{equation}
        We note that the right-hand side in \eqref{eq:discrete-compact} makes sense because $V_{h,t} = (\bP \bv_h, v_h) \in \bH^1_t \times H^1_\# = \bbV_{1,1}$.
        \begin{lemma}[properties of $\sA_h$] \label{lem:prop-of-sAh}
            The bilinear form $\sA_h$ defined in \eqref{eq:sAh} is continuous and satisfies a G\r{a}rding's inequality for all $W_h, V_h \in \bbV_h$:
            \[
            \begin{split}
                c_{\sA, 1} \|W_h\|_{1h,1}^2 - c_{\sA, 2} \|W_{h,t}\|_{0,0}^2 & \leq \sA_h[W_h, W_h] \\
                |\sA_h[W_h, V_h]| & \leq C_{\sA} \|W_h\|_{1h,1} \|V_h\|_{1h,1}.
            \end{split}
            \]
            where $c_{\sA, 1}, c_{\sA, 2}, C_{\sA} > 0$ come from Proposition \ref{prop:prop-of-sA} (properties of $\sA$).
        \end{lemma}
        \begin{proof}
            Using the properties of $b, c$ from \eqref{eq:prop-of-b-c}, the proof becomes virtually the same as that of Proposition \ref{prop:prop-of-sA}.
        \end{proof}
        \begin{remark}[relation to Stokes data regularity]
            The proposed method \eqref{eq:discrete} makes sense for all data $G = (\bg, g) \in \bbV_{-1,-1}$. From Corollary \ref{cor:reg-res-stokes} (higher regularity for Stokes), this is ensured when the data for the Stokes problem \eqref{eq:stokes} $F = (\bf, f)$ lies in $\bH^{-1}_t(\divG) \times H^1_\#$, where $\bH^{-1}_t(\divG)$ is defined in \eqref{eq:hdiv}. This is a stronger assumption than the minimal regularity data setting $F = (\bf, f) \in \bH^{-1}_t \times L^2_\#$ of the Stokes problem \eqref{eq:stokes}. This is due to the fact that the solution pair $(\bw_h, w_h)$ is to be controlled in $\bH^1_t \times H^1_\#$-norm as opposed to the more classical and weaker space $\bH^1_t \times L^2_\#$.
        \end{remark}

    \subsection{Interpolation} \label{sec:interp}
        We next show that the lifted spaces $\cV_h$ and $\bcV_h$ have optimal approximation properties. Let $V \in \bbV_{1,1}$ and $V_h \in \bbV_h$. We define the following \textit{error functional} for the $\bbV_{1,1}$-norm:
        \begin{equation} \label{eq:err-fun}
            \begin{split}
                \bE_{1h,1}[V, V_h] & := ( \|V - V_{h,t}\|_{1,1}^2 + \bk_h(\bv_h, \bv_h) )^{1/2} \\
                & \ = ( \|\nablaG (\bv - \bP \bv_h)\|^2 + \bk_h(\bv_h, \bv_h) + \|\nablaG (v - v_h)\| )^{1/2}.
            \end{split}
        \end{equation}
        In the sequel, we shall need basic interpolation error estimates for $v \in H^2(\G)$. For $d=2,3$, due to the embedding $H^2(\G) \embeds C^0(\G)$, one can use the Lagrange interpolant, cf. \cite{Dziuk1988, DziukElliott2013-2, Demlow2009}. For $d \geq 4$, we shall need an appropriately defined quasi-interpolant.  The construction of the Cl\'{e}ment-type \cite{Clement1975} operator was presented in \cite[\S 2.4]{DemlowDziuk2007} (for polyhedral surfaces); that of the Scott-Zhang \cite{ScottZhang1990} can be found in \cite[\S 4.4]{BonitoDemlowNochetto2020} (for polyhedral hypersurfaces) and in \cite[\S 3.1]{CamachoDemlow2015} (for higher-order geometric approximation). For the purposes of generality, we shall use the latter construction.
        
        Let $\widetilde{I}_h : H^1(\G_h) \to \widetilde{\cV}_h$ be the Scott-Zhang operator from \cite[\S 3.1]{CamachoDemlow2015}, and $\widetilde{\bI}_h : \bH^1(\G_h) \to \widetilde{\bcV}_h$ be its vectorial counterpart. We then define $I_h : H^1(\G) \to \cV_h$ and $\bI_h : \bH^1(\G) \to \bcV_h$, respectively, as
        \begin{equation} \label{eq:scott-zhang}
            I_h v := (\widetilde{I}_h v^{-l})^{l}, \qquad \bI_h \bv := (\widetilde{\bI}_h \bv^{-l})^{l},
        \end{equation}
        for all $(\bv, v) \in \bH^1(\G) \times H^1(\G)$. The lifting operators $(\cdot)^l, (\cdot)^{-l}$ are defined in \eqref{eq:lift}. The following lifted interpolation error estimate is an immediate consequence of the approximation result \cite[Theorem 3.2]{CamachoDemlow2015}, $H^2$-norm stability of the lifting operator \cite[Lemma 2.1]{CamachoDemlow2015}, and the equivalence results from \eqref{eq:lift-equiv}. We therefore omit the proof.
        \begin{proposition}[basic interpolation] \label{prop:basic-interp}
            Let $\G$ be of class $C^2$ and $v \in H^2(\G)$, then there exists constant $C_6 > 0$ depending only on $\G$ and shape-regularity such that
            \[
            \|v - I_h v\| + h \|\nablaG (v - I_h v)\| \leq C_6 h^2 \|v\|_2.
            \]
        \end{proposition}
        We next show that $\bI_h, I_h$ have optimal approximability with respect to the error functional $\bE_{1h,1}$.
        \begin{lemma}[interpolation error] \label{lem:interp}
            Let $\G$ be of class $C^3$, $V \in \bbV_{2,2}$ and $V_h := (\bI_h \bv, I_h v) \in \bbV_h$. Then there exists a constant $C_7 > 0$ depending only on $\G$ and shape-regularity such that
            \[\bE_{1h,1}[V, V_h] \leq C_7 h \|V\|_2.\]
        \end{lemma}
        \begin{proof}
            The error estimate for $v \in H^2_\#$ is trivial in light of Proposition \ref{prop:basic-interp} (basic interpolation). Next, let $\bv = (v_i)_{i=1}^{d+1} \in \bH^2_t \subset \bH^2$. From Proposition \ref{prop:basic-interp}, we infer that
            \begin{equation} \label{lem:interp:eq:err}
                \|\bv - \bI_h \bv\| + h\|\nabla_M(\bv - \bI_h \bv)\| \leq Ch^2 \|\bv\|_2,
            \end{equation}
            where $\nabla_M \bv = (\nabla_M^T v_i)_{i=1}^{d+1} = (\nablaG^T v_i)_{i=1}^{d+1}$. Using tangentiality of $\bv$, we next observe that
            \[\|\bv - \bI_h \bv\|^2 = \|\bP(\bv - \bI_h \bv)\|^2 + \|\bn \cdot (\bv - \bI_h \bv)\|^2 = \|\bv - \bP \bI_h \bv\|^2 + \|\bn \cdot \bI_h \bv\|^2.\]
            From the $L^2$-error estimate \eqref{lem:interp:eq:err} and the identity above, we have
            \[
            \|\bn \cdot \bI_h \bv\|^2 \leq \|\bv - \bI_h \bv\|^2 \leq Ch^4 \|\bv\|_2^2,
            \]
            whence
            \[
            \bk_h(\bI_h \bv, \bI_h \bv)^{1/2} \leq C h \|\bv\|_2.
            \]
            On the other hand, from the $H^1$-error estimate \eqref{lem:interp:eq:err}, we have
            \[
            \|\nablaG(\bv - \bI_h \bv)\| \leq \|\nabla_M(\bv - \bI_h \bv)\| \leq Ch \|\bv\|_2,
            \]
            because $\nablaG = \bP \nabla_M$. Next, we observe that
            \[
            \begin{split}
                \|\nablaG(\bv - \bP \bI_h \bv)\| & = \|\nablaG(\bP(\bv - \bI_h \bv))\| \\
                & \leq \|\nablaG \bP\|_{0,\infty} \|\bv - \bI_h \bv\| + \|\bP\|_{0,\infty}\|\nablaG(\bv - \bI_h \bv)\| \leq C h \|\bv\|_2.
            \end{split}
            \]
            Combining these estimates concludes the proof.
        \end{proof}
        
    \subsection{Quasi-orthogonality}      
        In view of definitions \eqref{eq:sA} of $\sA$ and \eqref{eq:sAh} of $\sA_h$, the continuous and discrete problems \eqref{eq:ell-pb}, \eqref{eq:discrete-compact} satisfy the following crucial Galerkin quasi-orthogonality property.
        \begin{lemma}[quasi-orthogonality] \label{lem:quasi-orth}
            If $W = (\bw, w) \in \bbV_{1,1}$ solves \eqref{eq:ell-pb} and $W_h = (\bw_h, w_h) \in \bbV_h$ solves \eqref{eq:discrete-compact}, then the following holds
            \[
            \begin{split}
                \sA[W - W_{h,t}, V_{h,t}] = \bk_h(\bw_h, \bv_h), \qquad \forall V_h = (\bv_h, v_h) \in \bbV_h.
            \end{split}
            \]
        \end{lemma}
        \begin{proof}
            Let $V_h = (\bv_h, v_h) \in \bbV_h$. We first observe that $V_{h,t} := (\bP \bv_h, v_h) \in \bH^1_t \times H^1_\# = \bbV_{1,1}$, hence $V_{h,t}$ is a valid test function for the continuous problem \eqref{eq:ell-pb}:
            \[
            \sA[W, V_{h,t}] = G[V_{h,t}].
            \]
            Since the discrete problem \eqref{eq:discrete-compact} reads $\sA_h[W_h, V_h] = G[V_{h,t}]$ for all $V_h \in \bbV_h$, we have
            \[\begin{split}
                0 & = \sA[W, V_{h,t}] - \sA_h[W_h, V_h] \\
                & = \sA[W, V_{h,t}] - \sA[W_{h,t}, V_{h,t}] - \bk_h(\bw_h, \bv_h) \\
                & = \sA[W - W_{h,t}, V_{h,t}] - \bk_h(\bw_h, \bv_h),
            \end{split}\]
            where on the second line we used \eqref{eq:sAh-2}. This proves the assertion.
        \end{proof}
        
    \subsection{Duality argument} \label{sec:duality-arg}
        In this section, we derive a conditional $L^2$-error estimate between $W$ and $W_h$ using the Aubin-Nitsche trick \cite{Nitsche1970}. Since we have not yet proved well-posedness of the discrete problem \eqref{eq:discrete-compact}, for the following result, we assume that \eqref{eq:discrete-compact} admits a solution $W_h \in \bbV_h$ for some $G \in \bbV_{-1,-1}$. This is a reasonable assumption since for $G = 0$, $W_h = 0$ is a solution. In the next section, we will use this result to prove discrete well-posedness for $h > 0$ small enough for any $G \in \bbV_{-1,-1}$, as first derived by Schatz \cite{Schatz1974}.
        \begin{lemma}[conditional $L^2$-error estimate] \label{lem:cond-l2-err-est}
            Assume that $\G$ is of class $C^3$ and data $G \in \bbV_{-1,-1}$. Let $W = (\bw, w) \in \bbV_{1,1}$ be the solution of \eqref{eq:ell-pb}, and suppose that $W_h = (\bw_h, w_h) \in \bbV_h$ is the solution of the discrete problem \eqref{eq:discrete-compact}. Then there exists a constant $C_8 > 0$ depending only on $\G$ and shape-regularity such that
            \begin{equation} \label{eq:err-est-in-l2}
                \|W - W_{h,t}\|_{0,0} \leq C_{8} \ h \ \bE_{1h,1}[W, W_h].
            \end{equation}
        \end{lemma}
        \begin{proof}
            Let $E := (\be_\bw, e_w) := (\bw - \bP \bw_h, w - w_h) = W - W_{h,t}$ be the error. Since $E \in \bbV_{1,1} \subset \bbV_{-1,-1}$, we consider the following adjoint problem: find $V = (\bv, v) \in \bbV_{1,1}$ such that for all $Y \in \bbV_{1,1}$:
            \begin{equation} \label{eq:dual-pb}
                \sA[Y, V] = E[Y].
            \end{equation}
            Note that since $E \in \bbV_{1,1}$, we can write the right-hand side of \eqref{eq:dual-pb} as
            \[
            E[Y] = \langle \be_\bw, \by \rangle_{-1,1} + \langle e_w, y \rangle_{-1,1} = (\be_\bw, \by) + (e_w, y).
            \]
            By Lemma \ref{lem:adjoint-regularity} (higher regularity of the adjoint problem) and the fact that $E \in \bbV_{1,1} \subset \bbV_{0,0}$, we have
            \begin{equation} \label{eq:dual-pb-reg}
                \|V\|_{2,2} \leq C_5 \|E\|_{0,0}.
            \end{equation}
            We test \eqref{eq:dual-pb} with $Y = E \in \bbV_{1,1}$ and use Lemma \ref{lem:quasi-orth} (quasi-orthogonality) for $V_h = (\bI_h \bv, I_h v)$, where $\bI_h, I_h$ is the Scott-Zhang interpolant (see section \ref{sec:interp}) to arrive at
            \[
            \begin{split}
                \|E\|_{0,0}^2 & = \sA[E, V] = \sA[W - W_{h,t}, V] \\
                & = \sA[W - W_{h,t}, V] - \sA[W - W_{h,t}, V_{h,t}] + \bk_h(\bw_h, \bv_h) \\
                & = \sA[W - W_{h,t}, V - V_{h,t}] + \bk_h(\bw_h, \bv_h).
            \end{split}
            \]
            Using Proposition \ref{prop:prop-of-sA} (properties of $\sA$), we obtain
            \[
            \begin{split}
                \|E\|_{0,0}^2 & \leq C_\sA \|W - W_{h,t}\|_1 \|V - V_{h,t}\|_1 + \bk_h(\bw_h, \bv_h) \\
                & \leq \max\{1, C_\sA\} \ \bE_{1h,1}[W, W_h] \ \bE_{1h,1}[V, V_h].
            \end{split}
            \]
            Next, we use Lemma \ref{lem:interp} (interpolation error) for $V$, followed by the $\bbV_{2,2}$-regularity estimate \eqref{eq:dual-pb-reg} for $V$ to see that
            \[
            \begin{split}
                \|E\|_{0,0}^2 & \leq \max\{1, C_\sA\} C_7 h \ \bE_{1h,1}[W, W_h] \ \|V\|_{2,2} \\
                & \leq \max\{1, C_\sA\} C_5 C_7 h \ \bE_{1h,1}[W, W_h] \ \|E\|_{0,0}.
            \end{split}
            \]
            Altogether, with $C_8 := \max\{1, C_\sA\} C_5 C_7$ we obtain the desired estimate \eqref{eq:err-est-in-l2}.
        \end{proof}
        
    \subsection{Well-posedness and quasi-optimality}
        We shall now prove a quasi-optimality result assuming that the discrete problem has a solution.
        \begin{lemma}[conditional quasi-optimality in $\bbV_{1,1}$] \label{lem:cond-quasi-opt}
            Assume that $\G$ is of class $C^3$ and data $G \in \bbV_{-1,-1}$. Let $W = (\bw, w) \in \bbV_{1,1}$ be the solution of \eqref{eq:ell-pb}, and suppose that $W_h = (\bw_h, w_h) \in \bbV_h$ is the solution of the discrete problem \eqref{eq:discrete-compact}. Then there exists a threshold $h_0 > 0$ given by
            \begin{equation} \label{eq:h0}
                h_0 := \sqrt{\frac{c_{\sA,1}}{4 c_{\sA,2} C_8^2}}
            \end{equation}
            such that for $0 < h \leq h_0$ and constant $C_9 = \frac{2 \max \{1, C_\sA^2\}}{c_{\sA,1}}$ we have
            \begin{equation} \label{eq:quasi-opt}
                \bE_{1h, 1}[W, W_h] \leq C_{9} \inf_{V_h \in \bbV_h} \bE_{1h,1}[W, V_h].
            \end{equation}
        \end{lemma}
        \begin{proof}
            Using G\r{a}rding's inequality from Proposition \ref{prop:prop-of-sA} (properties of $\sA$) applied to $E = W - W_{h,t} \in \bbV_{1,1}$, we have
            \[
            \begin{split}
                c_{\sA,1} \|E\|_{1,1}^2 - c_{\sA,2} \|E\|_{0,0}^2 \leq \sA[E, E].
            \end{split}
            \]
            Adding $c_{\sA, 1} \bk_h(\bw_h, \bw_h)$ to both sides and using $c_{\sA, 1} \leq 1$ results in
            \[
            \begin{split}
                c_{\sA,1} \bE_{1h, 1}[W, W_h]^2 - c_{\sA,2} \|E\|_{0,0}^2 & \leq \sA[E, E] + \bk_h(\bw_h, \bw_h) \\
                & = \sA[W - W_{h,t}, W - W_{h,t}] + \bk_h(\bw_h, \bw_h).
            \end{split}
            \]
            Let $V_h \in \bbV_h$ be a generic test function. Applying Lemma \ref{lem:quasi-orth} (quasi-orthogonality) with $W_h - V_h$ gives
            \[
            \begin{split}
                c_{\sA,1} \bE_{1h, 1}[W, W_h]^2 - c_{\sA,2} \|E\|_{0,0}^2 & = \sA[W - W_{h,t}, W - V_{h,t}] + \bk_h(\bw_h, \bv_h).
            \end{split}
            \]
            Invoking the continuity estimate for $\sA$ from Proposition \ref{prop:prop-of-sA} yields:
            \[
            \begin{split}
                c_{\sA,1} \bE_{1h, 1}[W, W_h]^2 - c_{\sA,2} \|E\|_{0,0}^2 & \leq \max\{1,C_{\sA}\} \ \bE_{1h,1}[W, W_h] \ \bE_{1h,1}[W, V_h] \\
                & \leq \frac{c_{\sA,1}}{2} \ \bE_{1h,1}[W, W_h]^2 + \frac{\max\{1,C_{\sA}^2\}}{2 c_{\sA,1}} \ \bE_{1h,1}[W, V_h]^2.
            \end{split}
            \]
            Therefore,
            \[
            \begin{split}
                & \frac{c_{\sA,1}}{2} \bE_{1h, 1}[W, W_h]^2 - c_{\sA, 2} \|W - W_{h,t}\|_{0,0}^2 \leq \frac{\max\{1,C_{\sA}^2\}}{2 c_{\sA,1}} \bE_{1h,1}[W, V_h]^2.
            \end{split}
            \]
            Using Lemma \ref{lem:cond-l2-err-est} (conditional $L^2$-error estimate) and letting
            \[
                h_0 := \sqrt{\frac{c_{\sA,1}}{4 c_{\sA,2} C_8^2}},
            \]
            we infer that the $L^2$-error can be absorbed into the first term, thus yielding:
            \[
            \frac{c_{\sA,1}}{4} \bE_{1h, 1}[W, W_h]^2 \leq \frac{\max\{1,C_{\sA}^2\}}{2 c_{\sA,1}} \bE_{1h,1}[W, V_h]^2.
            \]
            Defining $C_9 := \frac{2 \max\{1,C_{\sA}^2\}}{c_{\sA,1}^2}$ and recalling the fact that $V_h \in \bbV_h$ is arbitrary gives the desired \eqref{eq:quasi-opt}.
        \end{proof}
        
        The argument so far assumed that the discrete problem has a solution. Resorting to the fact that the discrete problem is linear, finite-dimensional, and square, we shall next show how this leads to discrete well-posedness, thereby making Lemma \ref{lem:cond-quasi-opt} (conditional quasi-optimality in $\bbV_{1,1}$) valid for any data $G \in \bbV_{-1,-1}$.

        \begin{theorem}[properties of discrete elliptic problem] \label{thm:discr-ell}
            Assume that $\G$ is of class $C^3$ and data $G \in \bbV_{-1,-1}$.
            Then for $0 < h \leq h_0$, with $h_0$ given by \eqref{eq:h0}, the discrete problem \eqref{eq:discrete} admits a unique solution $W_h \in \bbV_h$ and there exists $C_{10} > 0$ depending only on $\G$ and shape-regularity such that
            \begin{equation} \label{eq:discrete-stability}
                \|W_h\|_{1h,1} \leq C_{10} \|G\|_{-1,-1}.
            \end{equation}
            Furthermore, the following error estimates are valid:
            \begin{equation} \label{eq:error-est}
                \begin{split}
                    \bE_{1h, 1}[W, W_h] & \leq C_{9} \inf_{V_h \in \bbV_h} \bE_{1h,1}[W, V_h], \\
                    \|W - W_{h,t}\|_{0,0} & \leq C_{8} \ h \ \bE_{1h,1}[W, W_h],
                \end{split}
            \end{equation}
            with constants $C_8, C_9 > 0$ coming respectively from Lemmas \ref{lem:cond-l2-err-est} (conditional $L^2$-error estimate) and \ref{lem:cond-quasi-opt} (conditional quasi-optimality in $\bbV_{1,1}$).
        \end{theorem}
        \begin{proof}
            Since the problem is linear, finite-dimensional and square, it suffices to show that if $G = 0$, then $W_h = 0$. Let us first assume that a solution $W_h \in \bbV_h$ of \eqref{eq:discrete} exists for some $G \in \bbV_{-1,-1}$. Due to Lemma \ref{lem:cond-quasi-opt}, we have
            \[
                \bE_{1h, 1}[W, W_h] \leq C_{9} \inf_{V_h \in \bbV_h} \bE_{1h,1}[W, V_h] \leq C_{9} \bE_{1h,1}[W, 0] = C_9 \|W\|_{1,1}.
            \]
            It follows by the triangle inequality and Theorem \ref{thm:ell} (elliptic reformulation) that, if $C_{10} := C_3(1 + C_9)$, then
            \[
            \|W_h\|_{1h,1} \leq (1 + C_9) \|W\|_{1,1} \leq C_3 (1 + C_9) \|G\|_{-1,-1} = C_{10} \|G\|_{-1,-1}.
            \]
            Hence, if $G = 0$, then $W_h = 0$ because $\|\cdot\|_{1h,1}$ is a norm on $\bbV_h$. We thus infer that for $0 < h \leq h_0$, \eqref{eq:discrete} admits a unique solution $W_h \in \bbV_h$ for any $G \in \bbV_{-1,-1}$, with the stability constant $C_{10} > 0$ being independent of $h$. Finally, this also means that the estimate \eqref{eq:quasi-opt} from Lemma \ref{lem:cond-quasi-opt} as well as estimate \eqref{eq:err-est-in-l2} from Lemma \ref{lem:cond-l2-err-est} are no longer conditional but valid for any $G \in \bbV_{-1,-1}$.
        \end{proof}

        \begin{remark}[assumption on $h > 0$ being small]
            We stress that the smallness assumption on $h$ coming from \eqref{eq:h0} essentially means that the mesh size $h$ needs to resolve the deviation of the Ricci curvature $\bR = \tr(\bW) \bW - \bW^2$ of the surface $\G$ from the identity on $\G$ (i.e., $\bP$). This assumption is different from the usual condition that $h$ resolves all principal curvatures $\kappa_i$ of $\G$ \cite{OlshanskiiReuskenGrande2009, DziukElliott2013, BurmanHansboLarsonMassingZahedi2016, BonitoDemlow2019, BonitoDemlowNochetto2020}.
        \end{remark}

        We conclude this section with optimal order error estimates.

        \begin{theorem}[optimal error estimates]
            Assume that $\G$ is of class $C^3$ and data $G \in \bbV_{0,0}$. Let $W \in \bbV_{2,2}$ be the solution of \eqref{eq:ell-pb}, and $W_h \in \bbV_h$ be the solution of \eqref{eq:discrete-compact}. Then there exists $C_{11} > 0$ depending only on $\G$ and shape-regularity such that for $0 < h \leq h_0$ 
            \[
            \|W - W_{h,t}\|_{0,0} + h \ \bE_{1h,1}[W, W_h] \leq C_{11} h^2 \|G\|_{0,0}.
            \]
        \end{theorem}
        \begin{proof}
            We combine \eqref{eq:err-est-in-l2}, \eqref{eq:quasi-opt}, Lemma \ref{lem:interp} (interpolation error) together with Theorem \ref{thm:ell} (elliptic reformulation).
        \end{proof}

        \begin{remark}[higher-order error estimates]
            Upon resorting to higher regularity estimates for Stokes, and thus for the elliptic problem \eqref{eq:ell-pb}, as well as higher-order interpolation estimates, one can derive higher-order error estimates for the Galerkin solution of \eqref{eq:discrete-compact}.
        \end{remark}

        \begin{remark}[comparison with Stokes regularity]
            Recall from Corollary \ref{cor:reg-res-stokes} (higher regularity for Stokes) that the regularity assumption $G \in \bbV_{0,0}$ is guaranteed by $\bf \in \bH_t(\divG), f \in H^2_\#$, where $\bH_t(\divG)$ is defined in \eqref{eq:hdiv}. In contrast, $\bf \in \bL^2_t, f \in H^1_\#$ is the minimal data regularity for Stokes that leads to linear order of convergence in the space $\bH^1_t \times L^2_\#$. This regularity discrepancy is due to the fact that the current solution space is $\bH^1_t \times H^1_\#$, which turns out to be one order higher for the pressure.
        \end{remark}

\section{Numerical experiments} \label{sec:numer-experiments}
    In this section, we document that the proposed algorithm is both accurate and practical. We assess its accuracy on a non-trivial surface while using higher-order finite elements with equal and disparate polynomial degrees for velocity and pressure. We also apply the method to \eqref{eq:stokes-2-new}, which is governed by the modified surface Laplacian.

    We assume that $\G$ is the zero level set of a $C^3$ non-degenerate function $\phi$ and that $\bn = \nabla d$ is the outer unit normal vector to $\G$.

    \subsection{Discretization}
        We consider a shape-regular and quasi-uniform partition $\cT_h$ of $\G$ into either flat or curved triangles. For the algorithm, we assume that we have a continuous piecewise polynomial approximation of $\phi$ and $\bn$, denoted respectively $\phi_h, \bn_h$. Thus, we have a discontinuous polynomial approximation $\bW_h$ of the Weingarten map $\bW$, which in turn gives us $\bR_h$. By $\bn_h$ we shall denote the exact normal to the discrete surface $\G_h$, and set $\bP_h := \bI - \bn_h \otimes \bn_h$.

        As it is typically done in the literature \cite{OlshanskiiQuainiReuskenYushutin2018, HansboLarsonLarsson2020, JankuhnOlshanskiiReuskenZhiliakov2021}, we rewrite the elliptic part 
        \[
        \ba_{\G_h}(\bw_h, \bv_h) := (\nabla_{\G_h} (\bP_h \bw_h), \nabla_{\G_h} (\bP_h \bv_h))_{\G_h}
        \]
        as
        \[
        \tilde \ba_{\G_h} (\bw_h, \bv_h) := (\nabla_{\G_h} \bw_h - (\bn_h \cdot \bw_h) \bW_h, \nabla_{\G_h} \bv_h - (\bn_h \cdot \bv_h) \bW_h)_{\G_h}.
        \]
        The other bilinear forms in \eqref{eq:discrete} are adapted to $\G_h$ in a straightforward manner. We note that for the $\bk_h$-form, we ought to use a higher-order approximation $\bn_h$ of the normal $\bn$, as was shown in \cite{HansboLarsonLarsson2020, JankuhnOlshanskiiReuskenZhiliakov2021}. The right-hand side of \eqref{eq:discrete} is given, respectively, by
        \[
        (\bf^e, \bP_h \bv_h)_{\G_h} \qquad \text{and} \qquad (\bf^e + \bP_h \nabla g^e, \nabla_{\G_h} v_h)_{\G_h}.
        \]
        The condition that $\int_{\G_h} u_h \ds_h = \int_\G u \ds$ is imposed via a scalar Lagrange multiplier.

        The code is implemented in NGSolve \cite{Schoberl2009}, and all experiments were run on a commercial laptop Lenovo Legion 5 Pro 16ACH6H, with 16GB RAM and 16 threads. The monolithic linear system was solved using GMRes with a built-in geometric multigrid preconditioner. We report the textbook, $h$-independent, behavior of multigrid. The reproducible code is available on \href{https://github.com/chromomons/surface-stokes-without-inf-sup-condition}{https://github.com/chromomons/surface-stokes-without-inf-sup-condition}.

    \subsection{Computational order of convergence}
        We consider the surface $\G$ to be a 2-1 torus given by the zero levelset of
        \begin{equation}
            \phi(x,y,z) := (\sqrt{x^2 + y^2} - R)^2 + z^2 - r^2
        \end{equation}
        for $(R, r) = (2,1)$. Note that even though this surface is axisymmetric, its Gauss curvature changes sign (it is positive on the ``outside" and negative on the ``inside"). We consider the following velocity-pressure solution pair
        \begin{equation}
            \begin{split}
                \bu(x,y,z) & := \bP (-y^3, z^2+y, x)^T, \\
                u(x,y,z) & := \sin(\pi x) \sin(\pi y z),
            \end{split}
        \end{equation}
        noting that $\bu$ is not $\divG$-free and is not a polynomial due to $\bP$. The data $(\bf, f) = (-\Delta_B \bu + \nablaG u, \divG \bu)$ for the Stokes problem is computed in Wolfram Mathematica. For the algorithm we shall only need $\bf, \bf + \nablaG f$ and $\bar u := \int_\G u \ds$. We shall measure errors in the relative $L^2$ and $H^1$-errors over $\G_h$:
        \[
        \begin{split}
            \|e_w\|_0 & := \|w^e - w_h\|_{\G_h} / \|w\|, \\
            \|e_w\|_1 & := \|\nabla_{\G_h} (w^e - w_h)\|_{\G_h} / \|\nablaG w\|, \\
            \|\be_\bw\|_0 & := \|\bP_h (\bw^e - \bw_h)\|_{\G_h} / \|\bw\|, \\
            \|\be_\bw\|_1 & := \|\nabla_{\G_h} (\bw^e - \bw_h) - (\bn_h \cdot (\bw^e - \bw_h)) \bW_h\|_{\G_h} / \|\nablaG \bw\|.
        \end{split}
        \]
        Table \ref{tab:errors} depicts experimental orders of convergence for the equal-order-approximation of the velocity-pressure pair, which are known to be non-inf-sup-stable for the Stokes problem \eqref{eq:stokes-intro}. We report optimal rates. The longest run (the last refinement for $P_1$) took approximately 20 minutes, with the longest section of the code being the matrix assembly.
        
        In principle, velocity and pressure can be approximated with independent, arbitrary polynomial degrees. However, since the error between the two is coupled, according to Theorem \ref{thm:discr-ell} (properties of discrete elliptic problem), this would lead to suboptimal rates in the variable approximated with the higher-order polynomial degree. To investigate this question, we consider pairs $\bP_{k}/P_{k+r}, k, r \geq 1$ known to be non-inf-sup stable for \eqref{eq:stokes-intro}. The experimental error rates are reported in Table \ref{tab:errors-2}. For $k=1$, we observe a reduced rate in the $L^2$-error for the velocity, as expected. For $k \geq 2$, some of the error rates for the pressure appear to be better than expected.

    \subsection{Stokes with the surface diffusion operator}
        The proposed algorithm can be used to gain insight into the more physically relevant Stokes system \eqref{eq:stokes-2}. To fix ideas, consider the surface from \cite{Dziuk1988} given by the zero levelset of
        \[\phi(x,y,z) = (x-z^2)^2 + y^2 + z^2 -1. \]
        Let the data be given by $\bf(x,y,z) = (xy, \sin(xy^2), -\exp(-z))^T$ and $f = 0$ (divergence-free solution). As in \cite{OlshanskiiYushutin2019, BonitoDemlowLicht2020, OlshanskiiReuskenZhiliakov2021, DemlowNeilan2024}, we take $0 < \alpha \ll 1$, and next discretize \eqref{eq:stokes-2-new} using the proposed method with the $\bP_3/P_3$ velocity-pressure pair. The plot of the computed solution $(\bu_h, u_h)$ is depicted in Figure \ref{fig:dziuk}.

        \begin{table}[ht!]
            \captionsetup{width=.7\linewidth}
            \centering
            \small{
            \begin{tabular}{| c || c | >{\tt}c | | >{\tt}c | >{\tt}c || >{\tt}c | >{\tt}c || >{\tt}c | >{\tt}c || >{\tt}c | >{\tt}c ||}
                \hline
                \multirow{2}{*}{El.} & \multirow{2}{*}{$\ell$} & \multirow{2}{*}{dof num} & \multicolumn{2}{c||}{$\|\be_\bw\|_0$} & \multicolumn{2}{c||}{$\|\be_\bw\|_1$} & \multicolumn{2}{c||}{$\|e_w\|_0$} & \multicolumn{2}{c||}{$\|e_w\|_1$} \\
                \cline{4-11}
                & & & rate & error & rate & error & rate & error & rate & error \\
                \hline
                \hline
                \multirow{6}{*}{$\bP_1/P_1$}
                & $1$ & 1.67E+03 &      & 3.34E-02 &      & 1.59E-01 &      & 5.03E-01 &      & 7.68E-01 \\
                & $2$ & 6.40E+03 & 1.93 & 8.78E-03 & 0.94 & 8.28E-02 & 1.72 & 1.53E-01 & 0.90 & 4.12E-01 \\
                & $3$ & 2.51E+04 & 1.99 & 2.21E-03 & 0.99 & 4.18E-02 & 1.95 & 3.95E-02 & 0.96 & 2.12E-01 \\
                & $4$ & 9.95E+04 & 2.02 & 5.44E-04 & 1.00 & 2.09E-02 & 2.00 & 9.87E-03 & 1.00 & 1.06E-01 \\
                & $5$ & 3.95E+05 & 1.99 & 1.37E-04 & 1.01 & 1.04E-02 & 1.98 & 2.51E-03 & 0.99 & 5.36E-02 \\
                & $6$ & 1.57E+06 & 2.02 & 3.38E-05 & 1.00 & 5.19E-03 & 2.08 & 5.92E-04 & 1.04 & 2.60E-02 \\
                \hline 
                \hline
                \multirow{5}{*}{$\bP_2/P_2$}
                & $1$ & 6.69E+03 &      & 9.01E-04 &      & 1.37E-02 &      & 9.36E-02 &      & 2.85E-01 \\
                & $2$ & 2.56E+04 & 2.75 & 1.34E-04 & 1.69 & 4.25E-03 & 2.71 & 1.43E-02 & 1.67 & 8.98E-02 \\
                & $3$ & 1.00E+05 & 2.85 & 1.86E-05 & 1.78 & 1.24E-03 & 2.89 & 1.93E-03 & 1.91 & 2.39E-02 \\
                & $4$ & 3.98E+05 & 2.95 & 2.41E-06 & 1.93 & 3.26E-04 & 2.99 & 2.43E-04 & 1.99 & 5.99E-03 \\
                & $5$ & 1.58E+06 & 2.97 & 3.06E-07 & 1.97 & 8.30E-05 & 2.96 & 3.12E-05 & 1.98 & 1.52E-03 \\
                \hline 
                \hline
                \multirow{4}{*}{$\bP_3/P_3$}
                & $1$ & 1.50E+04 &      & 1.24E-04 &      & 3.17E-03 &      & 2.38E-02 &      & 1.13E-01 \\
                & $2$ & 5.76E+04 & 3.72 & 9.45E-06 & 2.55 & 5.42E-04 & 3.60 & 1.96E-03 & 2.57 & 1.90E-02 \\
                & $3$ & 2.26E+05 & 3.98 & 6.01E-07 & 2.91 & 7.21E-05 & 4.26 & 1.02E-04 & 3.13 & 2.17E-03 \\
                & $4$ & 8.95E+05 & 4.07 & 3.57E-08 & 3.08 & 8.51E-06 & 4.03 & 6.27E-06 & 3.01 & 2.69E-04 \\
                \hline
                \hline
                \multirow{3}{*}{$\bP_4/P_4$}
                & $1$ & 2.68E+04 &      & 2.34E-05 &      & 4.79E-04 &      & 1.06E-02 &      & 2.68E-02 \\
                & $2$ & 1.02E+05 & 5.03 & 7.14E-07 & 4.52 & 2.09E-05 & 5.04 & 3.22E-04 & 4.01 & 1.66E-03 \\
                & $3$ & 4.01E+05 & 4.98 & 2.27E-08 & 4.14 & 1.18E-06 & 5.01 & 9.97E-06 & 4.08 & 9.78E-05 \\
                \hline
                \hline
                \multirow{3}{*}{$\bP_5/P_5$}
                & $1$ & 4.18E+04 &      & 4.24E-06 &      & 1.30E-04 &      & 1.25E-03 &      & 8.36E-03 \\
                & $2$ & 1.60E+05 & 5.93 & 6.94E-08 & 5.10 & 3.81E-06 & 5.12 & 3.59E-05 & 4.25 & 4.41E-04 \\
                & $3$ & 6.27E+05 & 6.57 & 7.32E-10 & 5.30 & 9.64E-08 & 6.53 & 3.88E-07 & 5.40 & 1.04E-05 \\
                \hline
                \hline
                \multirow{3}{*}{$\bP_6/P_6$}
                & $1$ & 6.02E+04 &      & 9.09E-07 &      & 3.38E-05 &      & 5.05E-04 &      & 3.97E-03 \\
                & $2$ & 2.31E+05 & 6.86 & 7.80E-09 & 5.76 & 6.23E-07 & 6.86 & 4.34E-06 & 6.01 & 6.17E-05 \\
                & $3$ & 9.03E+05 & 8.19 & 2.66E-11 & 7.05 & 4.71E-09 & 7.54 & 2.33E-08 & 6.26 & 8.03E-07 \\
                \hline
            \end{tabular}}
            \caption{Experimental orders of convergence for $\bP_k/P_k, 1 \leq k \leq 6,$ velocity-pressure pairs. Here, $\ell \geq 1$ is the refinement level. The rates of convergence are consistent with theory.}
            \label{tab:errors}
        \end{table}

        \begin{table}[ht!]
            \captionsetup{width=.7\linewidth}
            \centering
            \small{
            \begin{tabular}{| c || c | >{\tt}c | | >{\tt}c | >{\tt}c || >{\tt}c | >{\tt}c || >{\tt}c | >{\tt}c || >{\tt}c | >{\tt}c ||}
                \hline
                \multirow{2}{*}{El.} & \multirow{2}{*}{$\ell$} & \multirow{2}{*}{dof num} & \multicolumn{2}{c||}{$\|\be_\bw\|_0$} & \multicolumn{2}{c||}{$\|\be_\bw\|_1$} & \multicolumn{2}{c||}{$\|e_w\|_0$} & \multicolumn{2}{c||}{$\|e_w\|_1$} \\
                \cline{4-11}
                & & & rate & error & rate & error & rate & error & rate & error \\
                \hline
                \hline
                \multirow{5}{*}{$\bP_1/P_2$}
                & 1 & 2.93E+03 &      & 2.79E-02 &      & 1.89E-01 &      & 1.28E-01 &      & 2.86E-01 \\
                & 2 & 1.12E+04 & 1.96 & 7.15E-03 & 0.94 & 9.89E-02 & 2.26 & 2.68E-02 & 1.67 & 9.01E-02 \\
                & 3 & 4.39E+04 & 1.99 & 1.80E-03 & 0.99 & 4.99E-02 & 2.14 & 6.08E-03 & 1.91 & 2.39E-02 \\
                & 4 & 1.74E+05 & 2.02 & 4.43E-04 & 1.00 & 2.49E-02 & 2.07 & 1.44E-03 & 1.99 & 6.00E-03 \\
                & 5 & 6.91E+05 & 2.00 & 1.11E-04 & 1.00 & 1.25E-02 & 2.01 & 3.57E-04 & 1.98 & 1.52E-03 \\
                \hline
                \hline
                \multirow{5}{*}{$\bP_1/P_3$}
                & 1 & 5.02E+03 &      & 2.79E-02 &      & 1.91E-01 &      & 9.13E-02 &      & 1.16E-01 \\
                & 2 & 1.92E+04 & 1.96 & 7.16E-03 & 0.95 & 9.92E-02 & 2.01 & 2.27E-02 & 2.52 & 2.01E-02 \\
                & 3 & 7.52E+04 & 1.99 & 1.80E-03 & 0.99 & 4.99E-02 & 1.98 & 5.76E-03 & 2.87 & 2.74E-03 \\
                & 4 & 2.98E+05 & 2.02 & 4.43E-04 & 1.00 & 2.49E-02 & 2.02 & 1.42E-03 & 2.47 & 4.95E-04 \\
                & 5 & 1.18E+06 & 2.00 & 1.11E-04 & 1.00 & 1.25E-02 & 2.00 & 3.56E-04 & 2.17 & 1.10E-04 \\
                \hline
                \hline
                \multirow{4}{*}{$\bP_2/P_3$}
                & 1 & 8.78E+03 &      & 1.35E-03 &      & 2.19E-02 &      & 2.38E-02 &      & 1.13E-01 \\
                & 2 & 3.36E+04 & 2.70 & 2.07E-04 & 1.71 & 6.70E-03 & 3.59 & 1.97E-03 & 2.57 & 1.90E-02 \\
                & 3 & 1.32E+05 & 2.83 & 2.92E-05 & 1.79 & 1.94E-03 & 4.26 & 1.03E-04 & 3.13 & 2.17E-03 \\
                & 4 & 5.22E+05 & 2.94 & 3.79E-06 & 1.93 & 5.08E-04 & 4.03 & 6.30E-06 & 3.01 & 2.69E-04 \\
                \hline
                \hline
                \multirow{4}{*}{$\bP_2/P_4$}
                & 1 & 1.17E+04 &      & 1.35E-03 &      & 2.20E-02 &      & 1.07E-02 &      & 2.68E-02 \\
                & 2 & 4.48E+04 & 2.70 & 2.08E-04 & 1.71 & 6.72E-03 & 4.92 & 3.52E-04 & 4.01 & 1.67E-03 \\
                & 3 & 1.76E+05 & 2.83 & 2.92E-05 & 1.79 & 1.94E-03 & 4.68 & 1.37E-05 & 4.04 & 1.01E-04 \\
                & 4 & 6.96E+05 & 2.94 & 3.79E-06 & 1.93 & 5.08E-04 & 4.27 & 7.09E-07 & 3.84 & 7.08E-06 \\
                \hline
                \hline
                \multirow{4}{*}{$\bP_2/P_5$}
                & 1 & 1.55E+04 &      & 1.35E-03 &      & 2.19E-02 &      & 1.86E-03 &      & 8.41E-03 \\
                & 2 & 5.92E+04 & 2.70 & 2.08E-04 & 1.71 & 6.71E-03 & 3.67 & 1.46E-04 & 4.16 & 4.72E-04 \\
                & 3 & 2.32E+05 & 2.83 & 2.92E-05 & 1.79 & 1.94E-03 & 3.96 & 9.39E-06 & 4.06 & 2.83E-05 \\
                & 4 & 9.20E+05 & 2.94 & 3.79E-06 & 1.93 & 5.08E-04 & 3.88 & 6.36E-07 & 3.02 & 3.50E-06 \\
                \hline
                \hline
                \multirow{4}{*}{$\bP_3/P_4$}
                & 1 & 1.80E+04 &      & 1.08E-04 &      & 3.11E-03 &      & 1.06E-02 &      & 2.68E-02 \\
                & 2 & 6.88E+04 & 3.61 & 8.82E-06 & 2.59 & 5.16E-04 & 5.04 & 3.22E-04 & 4.01 & 1.66E-03 \\
                & 3 & 2.70E+05 & 3.87 & 6.03E-07 & 2.84 & 7.21E-05 & 5.01 & 9.97E-06 & 4.08 & 9.78E-05 \\
                & 4 & 1.07E+06 & 3.99 & 3.79E-08 & 2.97 & 9.18E-06 & 4.99 & 3.13E-07 & 3.99 & 6.16E-06 \\
                \hline
            \end{tabular}}
            \caption{Experimental order of convergence for non-inf-sup-stable velocity-pressure pairs for \eqref{eq:stokes-intro} (with pressure having the higher polynomial degree). Here, $\ell \geq 1$ is the refinement level. The rates are either consistent with theory or better.}
            \label{tab:errors-2}
        \end{table}

        \begin{figure}[ht!]
            \centering
            \includegraphics[width=\linewidth]{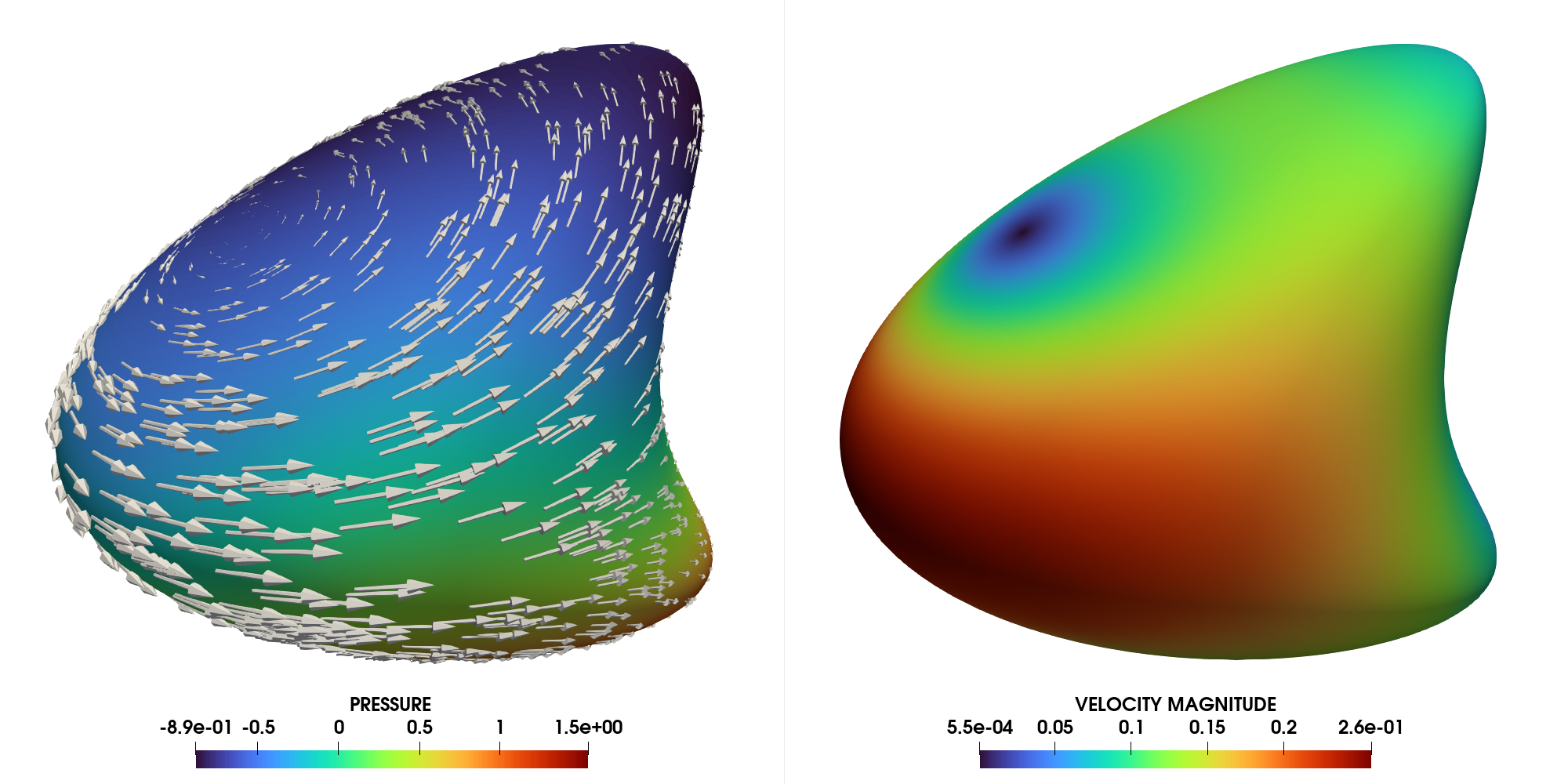}
            \caption{Numerical solution of the surface Stokes problem \eqref{eq:stokes-2} for $0 < \alpha \ll 1$ with $\bP_3/P_3$-element pair. The left picture depicts the velocity field colored by the pressure magnitude, while the right picture displays the velocity magnitude.}
            \label{fig:dziuk}
        \end{figure}
        
\section{Conclusions} \label{sec:conclusions}
    We develop a new finite element method for the surface Stokes problem. This approach hinges on an elliptic reformulation of the Surface Stokes problem in the primitive variables, allows for non-inf-sup stable velocity-pressure pairs for Stokes, and admits optimal order error estimates. We present a simple, idealized numerical analysis of a lifted parameteric FEM discretization of the problem for any polynomial degree. This method can be implemented and solved efficiently using geometric multigrid. We provide a numerical experiment revealing optimal order error estimates for higher-order elements. We also document the performance of the method for the more physically realistic Stokes model with the perturbed surface diffusion operator. Of course, the problem can be discretized with other methods of choice, like TraceFEM \cite{OlshanskiiQuainiReuskenYushutin2018, JankuhnOlshanskiiReuskenZhiliakov2021}, or penalty-free parametric FEMs \cite{BonitoDemlowLicht2020, DemlowNeilan2024, DemlowNeilan2025-2}.

    The main idea utilized here can also be applied to the transient Stokes problem. The corresponding manuscript is in preparation.

\printbibliography

\end{document}